\documentclass[11pt]{article}

\usepackage{amsfonts,amsmath,amsthm,amssymb}
\usepackage{bm}
\usepackage[colorlinks, citecolor=blue, urlcolor=blue, backref=page]{hyperref}
\usepackage[round]{natbib}
\usepackage[margin=1.2in]{geometry}
\usepackage[utf8]{inputenc}
\usepackage[T1]{fontenc}    

\usepackage{multirow}
\usepackage{pdflscape}
\usepackage{tabularx}

\usepackage{enumitem}

\newtheorem{theorem}{Theorem}[section]
\newtheorem{corollary}[theorem]{Corollary}
\newtheorem{lemma}[theorem]{Lemma}
\newtheorem{proposition}[theorem]{Proposition}

\newtheorem{example}[theorem]{Example}
\theoremstyle{definition}
\newtheorem{definition}[theorem]{Definition}

\newtheorem{remark}[theorem]{\textbf{Remark}}
\numberwithin{equation}{section}

\newcommand{\E}{{\mathbb E}}

\renewcommand{\P}{{\mathsf P}}
\newcommand{\Q}{{\mathsf Q}}
\newcommand{\R}{{\mathbb R}}

\newcommand{\N}{{\mathbb N}}

\newcommand{\Ecal}{{\mathcal E}}
\newcommand{\Fcal}{{\mathcal F}}

\newcommand{\Ncal}{{\mathcal N}}

\newcommand{\Pcal}{{\mathcal P}}

\newcommand{\Zcal}{{\mathcal Z}}

\newcommand{\Peff}{{\mathcal{P}_\textnormal{eff}}}

\DeclareMathOperator{\conv}{conv}

\begin{document}

\title{The numeraire e-variable and reverse information projection}
\author{
Martin Larsson\footnote{Department of Mathematical Sciences, Carnegie Mellon University;\texttt{larsson@cmu.edu}} \and
Aaditya Ramdas\footnote{Departments of Statistics and ML, Carnegie Mellon University; \texttt{aramdas@cmu.edu}} \and
Johannes Ruf\footnote{Department of Mathematics, London School of Economics; School of Data Science, Chinese University Hong Kong Shenzhen \texttt{j.ruf@lse.ac.uk}} \\[2ex]
}
\maketitle

\begin{abstract}
We consider testing a composite null hypothesis $\mathcal{P}$ against a point alternative $\mathsf{Q}$ using e-variables, which are nonnegative random variables $X$ such that $\mathbb{E}_\mathsf{P}[X] \leq 1$ for every $\mathsf{P} \in \mathcal{P}$. This paper establishes a fundamental result: under no conditions whatsoever on $\mathcal{P}$ or $\mathsf{Q}$, there exists a special e-variable $X^*$ that we call the numeraire, which is strictly positive and satisfies $\mathbb{E}_\mathsf{Q}[X/X^*] \leq 1$ for every other e-variable $X$. In particular, $X^*$ is log-optimal in the sense that $\mathbb{E}_\mathsf{Q}[\log(X/X^*)] \leq 0$. Moreover, $X^*$ identifies a particular sub-probability measure $\mathsf{P}^*$ via the density $d \mathsf{P}^*/d \mathsf{Q} = 1/X^*$. As a result, $X^*$ can be seen as a generalized likelihood ratio of $\mathsf{Q}$ against $\mathcal{P}$. We show that $\mathsf{P}^*$ coincides with the reverse information projection (RIPr) when additional assumptions are made that are required for the latter to exist. Thus $\mathsf{P}^*$ is a natural definition of the RIPr in the absence of any assumptions on $\mathcal{P}$ or $\mathsf{Q}$. In addition to the abstract theory, we provide several tools for finding the numeraire and RIPr in concrete cases. We discuss several nonparametric examples where we can indeed identify the numeraire and RIPr, despite not having a reference measure. Our results have interpretations outside of testing in that they yield the optimal Kelly bet against $\mathcal{P}$ if we believe reality follows $\mathsf{Q}$. We end with a more general optimality theory that goes beyond the ubiquitous logarithmic utility. We focus on certain power utilities, leading to reverse Rényi projections in place of the RIPr, which also always exist.
\end{abstract}

\section{Introduction}

Suppose we observe data from an unknown distribution on a measurable space $(\Omega,\Fcal)$. This paper is concerned with testing the composite null hypothesis $H_0$ that the data distribution belongs to a given set $\Pcal$ of distributions on this space, against the point alternative $H_1$ that the data was generated by a given distribution $\Q$. Here $\Pcal$ and $\Q$ are completely arbitrary, and our general theory will not require any further assumptions whatsoever.

We approach the testing problem by designing \emph{e-variables}: a $[0,\infty]$-valued  random variable $X$ is called an e-variable for $\Pcal$ if $\E_\P[X] \leq 1$ for all $\P \in \Pcal$. Specifically, we will look for `optimal' e-variables, in senses to be formalized later, that tend to be large under $\Q$.

E-variables measure evidence against the null: the larger its realized value, the stronger the evidence. One can also use e-variables to make hard accept/reject decisions. In particular, Markov's inequality implies that we can reject $H_0$ at level $\alpha$ if an e-variable $X$ exceeds $1/\alpha$, since $\P(X \geq 1/\alpha)\leq \alpha$ for each $\P \in \Pcal$. This paper will not be concerned with such hard decisions, but just in the measure of evidence $X$ itself, which requires no pre-specification of any level $\alpha$.

E-variables have been employed 
since at least~\cite{robbins1970statistical}. More recently, in 2018--2020, a number of papers appeared considering new applications and constructions of e-variables \citep{wasserman2020universal,howard2020time,howard2021time,shafer2021testing,vovk2021values,waudby2020estimating,grunwald2020safe}, but all using different (or no) terminology. In 2020, these authors agreed to use the term `e-variable' (and `e-value' for its realization). See~\cite{ramdas2023game} for a recent survey.

Our main result establishes that \emph{under absolutely no restrictions or conditions on $\Pcal$ or $\Q$}, there exists a special $\Q$-almost surely unique and strictly positive e-variable $X^*$ that we call the \emph{numeraire e-variable}, or just the \emph{numeraire}. It is characterized by the property that $\E_\Q[X/X^*] \le 1$  for every other e-variable $X$. This in turn implies two other interpretable properties $\E_\Q[X^*/X] \geq 1$ and $\E_\Q[\log (X/X^*)] \leq 0$, the second of which is log-optimality. In fact, $X^*$ is the numeraire if and only if it is log-optimal in this sense. The terminology derives from mathematical finance, where the \emph{numeraire portfolio} is a central object analogous to the numeraire e-variable; see e.g.\ \cite{LONG199029,MR1849424,Platen:2006, MR2335830, KK_book}. In the standard Neyman-Pearson framework with a composite null and alternative, already \cite{Cvitanic:Karatzas:2001} observed how techniques from financial mathematics can be transferred to obtain optimal statistical tests.

Although no assumptions on $\Pcal$ or $\Q$ are needed in general, the theory simplifies if the alternative $\Q$ does not assign positive probability to any event that is null under every $\P \in \Pcal$. In this case we say that \emph{$\Q$ is absolutely continuous with respect to $\Pcal$}, written $\Q \ll \Pcal$. This natural generalization of absolute continuity for pairs of measures goes back to \cite{halmos1949application} and is satisfied in most statistically relevant settings, while being significantly weaker than requiring a common dominating reference measure. 
\cite{Huber:Strassen:1973} provide early insights on statistical testing in the absence of a reference measure.

The condition $\Q \ll \Pcal$ has several consequences: (a) the numeraire is $\Q$-almost surely finite; (b) every e-variable is $\Q$-almost surely finite; and (c) the set of all e-variables for $\Pcal$ is bounded in probability under $\Q$, meaning that $\lim_{a \to \infty} \sup_X \Q(X > a) = 0$ where the supremum extends over all e-variables for $\Pcal$. In fact, we show that the above properties are equivalent in that they imply each other, and they also imply that $\Q \ll \Pcal$.

Next, we associate to the numeraire a special sub-probability measure $\P^* \ll \Q$ defined via its density $d\P^* / d\Q = 1/X^*$. This measure belongs to the \emph{effective null hypothesis} associated with $\Pcal$, denoted by $\Peff$. The effective null is the largest family of sub-probability measures that has the same set of e-variables as $\Pcal$. These are the distributions against which e-variables for $\Pcal$ do not provide any evidence.

We show that $\P^*$ can be characterized in terms of properties normally associated with the \emph{reverse information projection} (RIPr), first used by \cite{csiszar1984information}. Earlier work on the RIPr involve assumptions related to absolute continuity and/or finiteness of the minimum relative entropy between $\Q$ and $\Pcal$, and in such settings the above $\P^*$ coincides with the RIPr. This includes, in particular, work by~\cite{grunwald2020safe} and \cite{harremoes2023universal}, which in turn build on results in \cite{li1999estimation}. This leads us to argue that $\P^*$ is the natural definition of the RIPr in the absence of any assumptions on $\Pcal$ or $\Q$.

One of our main contributions is to significantly relax the conditions required for a log-optimality theory of e-variables to be built, and for a RIPr to exist. Furthermore, in addition to the abstract theory, we develop tools to aid in the computation of the numeraire and the RIPr in concrete examples. We establish several sufficient conditions to verify that an e-variable at hand is indeed the numeraire, and we provide several examples of nonparametric nulls $\Pcal$ (such as bounded distributions, symmetric distributions and sub-Gaussian distributions) that illustrate how these sufficient conditions can be used. These examples are out of the reach of previous works due to their lack of a reference measure on which the aforementioned works centrally rely.

An important message of our work is that in concrete examples, the duality between the set of e-variables and the null hypothesis can be exploited in different ways depending on the nature of the null. For parametric nulls, the simplest approach is usually to first find the RIPr (and then the numeraire). For non-parametric nulls, it is often easier to first find the numeraire (and then the RIPr).

We end the paper with a more general optimality theory that goes beyond the ubiquitous logarithmic utility, by studying other concave utilities and associated divergences. Without aiming for full generality, we focus on certain power utilities in place of the logarithm, which leads to \emph{reverse Rényi projections} in place of the RIPr. These developments, again, take place under absolutely no conditions on $\Pcal$ or $\Q$.

Finally, we note that if $\Pcal=\{\P\}$ is a singleton and $\Q \ll \P$, we obtain as a consequence that the numeraire is the likelihood ratio $d \Q/d \P$, which is well known to be log-optimal \citep{shafer2021testing}. Thus, our work can be  viewed as deriving certain generalized likelihood ratios for composite null hypotheses, and extending the corresponding optimality theory.

\paragraph{Paper outline.} Section~\ref{S_numeraire} defines the numeraire, establishes existence and uniqueness under no assumptions, characterizes the absolute continuity property $\Q \ll \Pcal$, and establishes various other useful facts such as log-optimality of the numeraire. Section~\ref{S_duality} introduces the RIPr and characterizes it via various optimality properties dual to the properties and log-optimality of the numeraire. Section~\ref{S_verification} establishes a number of tools for finding the numeraire and the RIPr in practice. These tools are then applied to specific examples in Section~\ref{S_examples}. Section~\ref{S_renyi} goes beyond log-optimality by considering certain power utilities and the associated reverse R\'enyi projections. Section~\ref{S_comp_alt_univ_inf} shows how our results imply that the method of universal inference is inadmissible. Section~\ref{S_summary} concludes. Appendix~\ref{app_gen_lebesgue} contains a generalized Lebesgue decomposition that we rely on in certain proofs, and Appendix~\ref{app_finiteness} gives an overview of the relation between absolute continuity properties and finiteness of the numeraire and relative entropies.

\paragraph{Notation.} Given the measurable space $(\Omega,\Fcal)$, we let $F_+$, $M_+$, and $M_1$ denote the set of all $[0,\infty]$-valued measurable functions, nonnegative measures, and probability measures on $\Fcal$, respectively. Given a null hypothesis $\Pcal$, which is an arbitrary nonempty subset of $M_1$, we write $\Ecal$ for the associated set of e-variables, i.e., the family of all $X \in F_+$ that satisfy $\E_\P[X] \leq 1$ for all $\P \in \Pcal$.
Infinite quantities play a key role in this paper. 
We generally rely on standard measure theoretic conventions such as $0 \cdot x = 0$ for $x \in [-\infty,\infty]$ and $x \cdot \infty = \infty$ for $x \in (0,\infty]$, as well as $\log(0) = -\infty$ and $\log(\infty) = \infty$. When evaluating ratios of $[0,\infty]$-valued random variables we often use the following convention, the last part of which is somewhat unusual but simplifies notation:
\begin{equation} \label{eq_arithmetic}
\text{$x/\infty = 0$ and $\infty/x = \infty$ for $x \in [0,\infty)$, and $\infty/\infty = 1$.}
\end{equation}
Clarifying explanations are given in potentially ambiguous cases. Finally, we write $\conv(\cdots)$ for convex hull, that is, the set of (finite) convex combinations of the indicated objects.

\section{The numeraire e-variable} \label{S_numeraire}

Here is the key definition of this paper. The convention~\eqref{eq_arithmetic} is used to evaluate the ratio inside the expectation.

\begin{definition} \label{D_numeraire}
A \emph{numeraire e-variable} (or just \emph{numeraire} for short) is a $\Q$-almost surely strictly positive e-variable $X^*$ such that $\E_\Q[X/X^*] \le 1$ for every e-variable $X$.
\end{definition}

Our choice of terminology is inspired by the \emph{numeraire portfolio} in the mathematical finance literature; see e.g.\ \cite{LONG199029,MR1849424, Platen:2006, MR2335830, KK_book}. In that literature, the defining property of the numeraire portfolio is that any positive wealth process divided by the numeraire becomes a supermartingale. In view of the gambling interpretation of e-variables \citep{shafer2021testing}, this is precisely what is expressed by Definition~\ref{D_numeraire}. A subtle distinction is that we allow the numeraire to take the value infinity, whereas this is excluded in the finance literature.

Numeraires are unique up to $\Q$-nullsets.  Indeed, if $X^*_1$ and $X^*_2$ are numeraires, then the ratio $Y = X^*_2/X^*_1$ satisfies $1 \le 1/\E_\Q[Y] \le \E_\Q[1/Y] \le 1$ thanks to the numeraire property of $X^*_1$, Jensen's inequality, and the numeraire property of $X^*_2$. Thus Jensen's inequality holds with equality, so $Y$ is $\Q$-almost surely equal to a constant which must be one. It follows that $X^*_1$ and $X^*_2$ are $\Q$-almost surely equal. In view of this uniqueness, we often speak of `the' numeraire.

\begin{example} \label{Ex:240225}
In the case of a simple null $\Pcal = \{\P_0\}$ with $\Q$ absolutely continuous with respect to $\P_0$, the numeraire is just the likelihood ratio $X^* = d\Q/d\P_0$. Indeed, $X^*$ is $\Q$-almost surely strictly positive, it is an e-variable, and for any other e-variable $X$ we have $\E_\Q[X/X^*] = \E_{\P_0}[X] \le 1$. In particular settings, such as when a reference measure exists and certain additional assumptions are satisfied, we will show that the numeraire is a likelihood ratio involving the `reverse information projection'  \citep{grunwald2020safe,harremoes2023universal}. Even without these restrictions, $X^*$ can still be interpreted as a likelihood ratio. One of our goals is to establish this in full generality. This is the subject of Section~\ref{S_duality}. 
\end{example}

The following general property of the numeraire can sometimes be useful.

\begin{lemma}
\label{lem:lifting}
Let $X^*$ denote the numeraire for $\Pcal$, and consider a larger null hypothesis $\Pcal' \supset \Pcal$. If $X^*$ is still an e-variable for $\Pcal'$, then it is also the numeraire for $\Pcal'$.
\end{lemma}
\begin{proof}
If $X$ is an e-variable for $\Pcal'$, it is also an e-variable for $\Pcal$ and one has $\E_\Q[X/X^*] \leq 1$ by assumption. This proves the lemma.
\end{proof}

A numeraire $X^*$ is also \emph{log-optimal} in the sense that
\begin{equation} \label{eq_log_optimality}
\text{$\E_\Q\left[ \log \frac{X}{X^*} \right] \le 0$ for every e-variable $X$,}
\end{equation}
where the left-hand side may be $-\infty$. This follows directly from Jensen's inequality and the numeraire property. The converse is also true, and we record the equivalence in the following proposition; see also \citet[Theorem~16.2.2]{MR2239987}. The proof shows that the numeraire property is really the first-order condition for log-optimality. Note also that a numeraire $X^*$ is the $\Q$-almost surely unique log-optimal e-variable in the sense of \eqref{eq_log_optimality} even if $\E_\Q[\log X^*]$ happens to be infinite.

\begin{proposition} \label{P_log_optimality}
Let $X^*$ be a $\Q$-almost surely strictly positive e-variable. Then $X^*$ is a numeraire if and only if it is log-optimal. In particular, a log-optimal e-variable is unique up to $\Q$-nullsets.
\end{proposition}

\begin{proof}
The forward direction was argued above. To prove the converse we assume $X^*$ is log-optimal. For any e-variable $X$ and $t \in (0,1)$, $X(t) = (1-t) X^* + t X$ is an e-variable. Thus by log-optimality, $\E_\Q[t^{-1}\log(X(t)/X^*)] \le 0$. The expression inside the expectation equals $t^{-1} \log(1 - t + t X/X^*)$,
which converges to $X/X^* - 1$ as $t$ tends to zero and is bounded below by $t^{-1}\log(1-t)$, hence by $-2 \log 2$ for $t \in (0, 1/2)$. Fatou's lemma thus yields $\E_\Q[X/X^* - 1] \le 0$, showing that $X^*$ is a numeraire. Finally, the uniqueness statement follows from the equivalence just proved together with uniqueness of the numeraire up to $\Q$-nullsets.
\end{proof}

\begin{remark} \label{R_benchmark_numeraire}
In some situations, for instance if $\log X^*$ has infinite $\Q$-expectation, it may be of interest measure evidence relative to a \emph{benchmark}. By a benchmark we simply mean any $\Q$-almost surely strictly positive random variable $X_\text{b}$ such that $\log(X^*/X_\text{b})$ has finite $\Q$-expectation. Note that the numeraire itself is a possible benchmark. For any $X \in \Ecal$, the log-optimality property \eqref{eq_log_optimality} of $X^*$ yields 
\[
\E_\Q\left[\log \frac{X}{X_\text{b}}\right] = \E_\Q\left[\log \frac{X}{X^*}\right] + \E_\Q\left[\log \frac{X^*}{X_\text{b}}\right] \le \E_\Q\left[\log \frac{X^*}{X_\text{b}}\right].
\]
Thus the numeraire maximizes the `benchmarked' quantity $\E_\Q[\log(X/X_\text{b})]$ over $\Ecal$, and the optimal value is finite. This is true regardless of which benchmark is used, so the numeraire is \emph{benchmark-invariant}. We thank an anonymous referee for pointing this out.
\end{remark}

The following theorem shows that a numeraire always exists without any assumptions whatsoever on $\Pcal$ or $\Q$. The theorem is related to an important family of results in the mathematical finance literature, known collectively as the Fundamental Theorem of Asset Pricing, which are versions of the statement that a numeraire portfolio with finite wealth exists if and only if unbounded profits must come with unbounded downside risk. The following theorem contains the corresponding statement in the context of e-variables.

\begin{theorem} \label{T_numeraire}
A numeraire exists. Moreover, the following conditions are equivalent:
\begin{enumerate}
\item\label{T_numeraire_1} The numeraire is $\Q$-almost surely finite.
\item\label{T_numeraire_2} Every e-variable is $\Q$-almost surely finite.
\item\label{T_numeraire_3} $\Q$ is absolutely continuous with respect to $\Pcal$.
\item\label{T_numeraire_4} The set of e-variables $\Ecal$ is bounded in probability under $\Q$.
\end{enumerate}
\end{theorem}

\begin{proof}
The existence of a numeraire is Lemma~\ref{L_numeraire_exists} below, which relies on the preceding Lemmas~\ref{L_komlos} and~\ref{L_u_max}. Here we focus on the equivalence of the four conditions. Clearly \ref{T_numeraire_2} implies \ref{T_numeraire_1}. To see that \ref{T_numeraire_1} implies \ref{T_numeraire_2}, let $X^*$ be a $\Q$-almost surely finite numeraire and note that the numeraire property implies that any e-variable $X$ must be $\Q$-almost surely finite too. Next, we claim that \ref{T_numeraire_2} is equivalent to \ref{T_numeraire_3}. Suppose \ref{T_numeraire_3} fails. Then there is an event $A$ with $\Q(A) > 0$ and $\P(A) = 0$ for all $\P \in \Pcal$, in which case $X = \infty \bm1_A$ is an e-variable that is not $\Q$-almost surely finite. Thus \ref{T_numeraire_2} fails also. The converse  follows by noting that every e-variable is finite $\P$-almost surely for all $\P \in \Pcal$. Further, \ref{T_numeraire_4} implies \ref{T_numeraire_2} because under the former condition we have for any e-variable $X$ that $\Q(X = \infty) = \lim_{a \to \infty}\Q(X > a) = 0$. It remains to prove that \ref{T_numeraire_2} implies \ref{T_numeraire_4}. This more involved implication is the content of Lemma~\ref{L_unbounded_in_prob} below.
\end{proof}

Conditions~\ref{T_numeraire_2} and~\ref{T_numeraire_4} of the theorem describe the existence of a finite numeraire in financial terms. Condition~\ref{T_numeraire_2} simply states that infinite profits are impossible. Condition~\ref{T_numeraire_4} is a `pre-limiting' expression of the same fact, in the sense that it excludes sequences of bets that achieve unbounded profits with a fixed positive probability. Condition~\ref{T_numeraire_3} has statistical rather than financial meaning. If it were violated, the alternative would assign positive probability to some event that has zero probability under every element of the null; if such an event occurs, the null can be immediately rejected at any level, meaning that any e-variable could be set to equal infinity on that event, violating the other statements.

The rest of this section contains the remaining parts of the proof of Theorem~\ref{T_numeraire}. This material can be skipped without compromising the understanding of Sections~\ref{S_duality}, \ref{S_verification}, or \ref{S_examples}.

\begin{lemma} \label{L_unbounded_in_prob}
If every e-variable is $\Q$-almost surely finite then $\Ecal$ is bounded in probability under $\Q$.
\end{lemma}

\begin{proof}
Consider the space $L^0_+(\Q)$ consisting of all nonnegative real-valued random variables modulo $\Q$-almost sure equivalence. Let $C \subset L^0_+(\Q)$ be the set of (equivalence classes of) random variables that are $\Q$-almost surely equal to a finite e-variable. The set $C$ is convex, solid, and closed in probability. Suppose now that $\Ecal$ is not bounded in probability under $\Q$. Then $C$ is not bounded in probability. The decomposition result in Lemma~2.3 of \cite{Brannath}   then shows that there is an event $A$ with $\Q(A) > 0$ such that $n \bm1_A$ belongs to $C$ for all $n \in \N$. By definition of $C$, for each $n$ there is a finite e-variable $X_n$ that is $\Q$-almost surely equal to $n \bm1_A$, say on an event $B_n$ with $\Q(B_n) = 1$. By Fatou's lemma, $X = \liminf_{n\to \infty} X_n$ is again an e-variable. Moreover, $X = \infty$ on $A \cap (\bigcap_{n \in \mathbb{N}} B_n)$, which coincides with $A$ up to a $\Q$-nullset and thus has positive probability under $\Q$. In other words, $X$ is an e-variable that is not $\Q$-almost surely finite.
\end{proof}

The following result, usually referred to as a `Koml\'os-type' lemma, is the basis of the proof that a numeraire exists. For a proof, see Lemma~A1.1 and the subsequent Remark~1 in \cite{DS:1994}.

\begin{lemma} \label{L_komlos}
Let $(X_n)_{n \in \N}$ be any sequence of $[0,\infty]$-valued random variables. One can choose $\widetilde X_n \in \conv(X_n, X_{n+1}, \ldots)$ such that $\widetilde X_n$ converges $\Q$-almost surely to a $[0,\infty]$-valued random variable.
\end{lemma}

Next, we consider a certain concave maximization problem. This involves a \emph{utility function}, by which we mean a function $U\colon (0,\infty) \to \R$ that is continuous, nondecreasing, concave, and differentiable. By convention we extend the domain of $U$ to $[0,\infty]$ by setting $U(0) = \lim_{x \to 0} U(x)$ and $U(\infty) = \lim_{x \to \infty} U(x)$, and we define $U'(0)$ and $U'(\infty)$ similarly. These limits exist by monotonicity and concavity of $U$, and could be $\pm \infty$. The following result assumes in addition that $U$ is bounded above, which forces $U'(\infty) = 0$.

\begin{lemma} \label{L_u_max}
Let $U$ be a utility function that is bounded from above. There exists an e-variable $X^*$ that attains $\sup_{X \in \Ecal} \E_\Q[ U(X) ]$ and has the property that every e-variable is $\Q$-almost surely finite on $\{X^* < \infty\}$. If $\E_\Q[ U'(X^*) X^* ] < \infty$, one has the first-order condition
\begin{equation} \label{eq_L_u_max_FOC}
\E_\Q\left[ U'(X^*) (X - X^*) \right] \le 0, \quad X \in \Ecal.
\end{equation}
Here $U'(X^*)X^*$ and $U'(X^*) (X - X^*)$ are understood as zero on $\{X^* = \infty\}$.
\end{lemma}

\begin{proof}
Let $X_n \in \Ecal$ form a sequence such that $\E_\Q[U(X_n)]$ increases to the supremum. Lemma~\ref{L_komlos} yields a $\Q$-almost surely convergent sequence of random variables of the form $\widetilde X_n = \sum_{i=n}^{m_n} \lambda^n_i X_i$ for convex weights $\lambda^n_n,\ldots,\lambda^n_{m_n}$. Define $X^* = \liminf_{n\to\infty} \widetilde X_n$ and note that while this may be a proper limit inferior for some $\omega \in \Omega$, we have $\Q$-almost surely that $X^* = \lim_{n\to \infty} \widetilde X_n$. Each $\widetilde X_n$ belongs to $\Ecal$ since this set is convex, and Fatou's lemma then implies that $X^*$ belongs to $\Ecal$ as well. Next, since $U$ is concave and $\E[U(X_n)]$ is an increasing sequence,
\begin{equation} \label{eq_L_u_max_1}
\E_\Q[U(\widetilde X_n)] \ge \sum_{i=n}^{m_n} \lambda^n_i \E_\Q[U(X_i)] \ge \E_\Q[U(X_n)].
\end{equation}
Since $U$ is bounded above and $X^* = \lim_{n\to \infty} \widetilde X_n$ $\Q$-almost surely, we may use Fatou's lemma together with \eqref{eq_L_u_max_1} to get
\[
\E_\Q[U(X^*)] \ge \limsup_{n \to \infty} \E_\Q[U(\widetilde X_n)] \ge \limsup_{n \to \infty} \E_\Q[U(X_n)] = \sup_{X \in \Ecal} \E_\Q[U(X)].
\]
This shows that $X^*$ attains the supremum.

We next show that $X^*$ can be chosen so that every e-variable is $\Q$-almost surely finite on $\{X^* < \infty\}$. This uses Lemma~\ref{L_gen_Lebesgue}, which in particular states that there is an event $N$ that is $\Pcal$-negligible, meaning that $\P(N) = 0$ for all $\P \in \Pcal$, and satisfies $\Q(N' \setminus N) = 0$ for any other $\Pcal$-negligible event $N'$. Define $X^{**} = X^*\bm1_{N^c} + \infty \bm1_N$. This is still an e-variable since $N$ is $\Pcal$-negligible, and it attains the supremum since it is at least as large as $X^*$. Furthermore, since $N' = \{X = \infty\}$ is $\Pcal$-negligible for any $X \in \Ecal$, we have $\Q(\{X=\infty\} \setminus \{X^{**}=\infty\}) = 0$. Equivalently, by taking complements, $X$ is $\Q$-almost surely finite on $\{X^{**} < \infty\}$. Replacing $X^*$ by $X^{**}$ thus ensures the claimed property.

We now verify the first-order condition \eqref{eq_L_u_max_FOC}. Pick $X \in \Ecal$ such that $X \ge \delta$ for some $\delta > 0$; this restriction will be removed later. Define $X(t) = t X + (1-t) X^*$ for $t \in [0,1]$ and use the optimality of $X^* = X(0)$ to get
\[
\E_\Q\left[ \frac{U(X(t)) - U(X(0))}{t} \right] \le 0, \quad t \in (0,1].
\]
By concavity of $t \mapsto U(X(t))$, as $t$ tends to zero the difference quotient increases to $\frac{d}{dt}|_{t=0} U(X(t)) = U'(X^*)(X - X^*)$, understood to be zero on $\{X^* = \infty\}$, and is bounded below by $U(X(1)) - U(X(0)) \ge U(\delta) - \sup_{x > 0} U(x)$. Monotone convergence then yields \eqref{eq_L_u_max_FOC} for all $X$ bounded below by $\delta$. We now remove this restriction and consider an arbitrary $X \in \Ecal$. For each $\delta \in (0,1)$ the random variable $\delta + (1-\delta)X$ still belongs to $\Ecal$ and is bounded below by $\delta$. By what we just proved,
\[
\E_\Q\left[ U'(X^*) (\delta + (1-\delta)X - X^*) \right] \le 0.
\]
The expression inside the expectation is bounded below by $-U'(X^*) X^*$, understood to be zero on $\{X^*=\infty\}$, and this is integrable by assumption. We may then send $\delta$ to zero and use Fatou's lemma to deduce \eqref{eq_L_u_max_FOC}.
\end{proof}

\begin{lemma} \label{L_numeraire_exists}
A numeraire exists.
\end{lemma}

\begin{proof}
If we could have taken $U(x) = \log(x)$ in Lemma~\ref{L_u_max}, the first-order condition \eqref{eq_L_u_max_FOC} would have been the numeraire property. We cannot do this directly, because the logarithm is not bounded above. Instead we consider bounded approximations of the logarithm, each giving an associated optimal e-variable and first-order condition. We then pass to the limit in the first-order conditions using the Koml\'os-type lemma. We now turn to the details.

For each $n \in \N$ we define $U_n(x)$ by letting $U_n(1) = 0$ and setting the derivative to
\[
U_n'(x) = \begin{cases} x^{-1}, & x \le n, \\ 2n^{-1} - xn^{-2}, & n < x \le 2n, \\ 0, & x > 2n. \end{cases}
\]
Then $U_n$ satisfies the properties required in Lemma~\ref{L_u_max}, and we record for later use that $U_n'$ is convex and that
\begin{equation} \label{eq_P_numeraire_iii_to_i_2}
\text{$U_n'(x) \le U_i'(x)$ for $n \le i$ and $x \in (0,\infty]$}.
\end{equation}
Applying Lemma~\ref{L_u_max} with $U = U_n$ gives an optimal e-variable $X_n^*$ for each $n$. These have the additional property that the sets $\{X_n^* < \infty\}$ all coincide up to $\Q$-nullsets, say with a fixed set $A$. Because $U'_n(x) x \le 1$ for all $x \in (0,\infty]$, we may use the first-order condition \eqref{eq_L_u_max_FOC} to deduce
\begin{equation} \label{eq_P_numeraire_iii_to_i_3}
\E_\Q\left[ U_n'(X_n^*) X \bm1_A \right] \le \Q(A), \quad X \in \Ecal.
\end{equation}
Next, Lemma~\ref{L_komlos} yields a $\Q$-almost surely convergent sequence of random variables $\widetilde X_n^* = \sum_{i=n}^{m_n} \lambda^n_i X_i^*$ for some convex weights $\lambda^n_n,\ldots,\lambda^n_{m_n}$. We define $X^* = \liminf_{n \to \infty} \widetilde X_n^*$, which belongs to $\Ecal$ and is $\Q$-almost surely the limit of $\widetilde X_n^*$. Note that $A$ then coincides $\Q$-almost surely with the sets where $X^*$ and the $\widetilde X^*_n$ are finite. Using first \eqref{eq_P_numeraire_iii_to_i_3}, then \eqref{eq_P_numeraire_iii_to_i_2}, and finally the convexity of $U'_n$ we get 
\begin{align*}
\Q(A) \ge \E_\Q\left[ \sum_{i=n}^{m_n} \lambda^n_i U_i'(X_i^*) X \bm1_A \right] \ge \E_\Q\left[ \sum_{i=n}^{m_n} \lambda^n_i U_n'(X_i^*) X \bm1_A \right] \ge \E_\Q\left[ U_n'(\widetilde X_n^*) X \bm1_A \right].
\end{align*}
The expression inside the last expectation is nonnegative and converges $\Q$-almost surely to $X \bm1_A / X^*$. In particular, $X^*$ is $\Q$-almost surely strictly positive on $A$. Fatou's lemma now yields $\E_\Q[X \bm1_A /X^*] \le \Q(A)$. On $A^c$, where $X^* = \infty$, we have $X/X^* \le 1$ thanks to the convention \eqref{eq_arithmetic}. Hence $\E_\Q[X/X^*] \le \E_\Q[X \bm1_A /X^*] +  \Q(A^c) \le 1$. This shows that $X^*$ is a numeraire.
\end{proof}

\section{Duality and reverse information projection} \label{S_duality}

We continue to consider an arbitrary nonempty composite null hypothesis $\Pcal$ and a simple alternative $\Q$ on a measurable space $(\Omega,\Fcal)$. There is a duality relationship between $\Pcal$ and its set of e-variables $\Ecal$. To develop this perspective we make the following crucial definition, where we recall that $M_+$ denotes the set of all nonnegative measures on $\Omega$.

\begin{definition}
The \emph{effective null hypothesis} is the set
\[
\Peff = \{\P \in M_+ \colon \E_\P[X] \le 1 \text{ for all } X \in \Ecal\}.
\]
\end{definition}

The effective null hypothesis is the largest family of distributions whose set of e-variables is exactly $\Ecal$. Thus it consists of those distributions against which e-variables in $\Ecal$ do not provide any evidence. In particular, a `nontrivial' e-variable $X$, in the sense that $\E_\Q[X] > 1$, exists if and only if $\Q \notin \Peff$. 
Because the constant $X = 1$ is always an e-variable, elements of $\Peff$ have total mass at most one, and sometimes strictly less than one. For example, if $\P$ belongs to $\Peff$, then so does any $\P'$ in $M_+$ that is set-wise dominated by $\P$, meaning that $\P'(A) \le \P(A)$ for all $A$. 
In Section~\ref{S_verification} we discuss more explicit descriptions of the effective null; see  Example~\ref{ex_P_bipolar}, Lemma~\ref{L_P_bipolar_char_ref_measure}, and Corollary~\ref{cor:conv-hull}. See also Example~\ref{ex_simple_null} below.

\begin{remark}
There is a convex analytic interpretation of $\Ecal$ and $\Peff$ as the so-called \emph{polar} and \emph{bipolar} of $\Pcal$, usually denoted by $\Pcal^\circ$ and $\Pcal^{\circ\circ}$.\footnote{This terminology is slightly nonstandard in that $F_+$, which appears in the definition of $\Ecal$ (see the notation paragraph in the introduction), is not a subset of a locally convex topological vector space, whereas the standard theory of dual pairs works with such spaces. There are however exceptions, for example the bipolar theorems of \cite{Brannath} and \cite{MR3910414}. 
} For some of our results we make use of this connection by invoking a suitable \emph{bipolar theorem}; see the proof of Lemma~\ref{L_P_bipolar_char_ref_measure}.
\end{remark}

The following observation is simple, yet key to the theory developed in this paper.

\begin{lemma} \label{L_P_star}
The numeraire e-variable $X^*$ gives rise to an element $\P^* \in \Peff$ defined by $d\P^*/d\Q = 1/X^*$, understood as zero on $\{X^* = \infty\}$. Note that $\P^* \ll \Q$ by definition. 
\end{lemma}

\begin{proof}
Since $X^*$ is $\Q$-almost surely unique and strictly positive, $\P^*$ is well-defined. It belongs to $\Peff$ because $\E_{\P^*}[X] = \E_\Q[\bm1_{\{X^* < \infty\}}X/X^*] \le \E_\Q[X/X^*] \le 1$ for all $X \in \Ecal = \Pcal^\circ$ thanks to the numeraire property.
\end{proof}

Our results show that $\P^*$ can be interpreted as a \emph{reverse information projection}, as we explain after Theorem~\ref{T_duality} below. This concept first appeared in \cite{csiszar1984information} and later in the PhD thesis of \citet[Chapter 4]{li1999estimation} as well as in \cite{1201070}. It was used centrally in the context of e-variables in \cite{grunwald2020safe}, and was developed further in a recent preprint by~\cite{harremoes2023universal}. This earlier work assumes the existence of a reference measure, an assumption we are able to dispense with but nonetheless examine further in Section~\ref{S_verification}.

\begin{definition} \label{D_RIPr}
We refer to the measure $\P^*$ in Lemma~\ref{L_P_star} as the \emph{reverse information projection} (\textnormal{RIPr}) of $\Q$ on $\Peff$.
\end{definition}

Our use of the RIPr terminology for $\P^*$ is justified by what follows.
Recall that the entropy of $\Q$ relative to $\P$ is defined by
\[
H(\Q \mid \P) = \begin{cases}
\E_\Q\left[ \log \dfrac{d\Q}{d\P} \right] & \text{if } \Q \ll \P, \\
+\infty & \text{otherwise.}
\end{cases}
\]
We will show that $\P^*$ minimizes the relative entropy $H(\Q \mid \P)$ over $\Peff$. For this it will be convenient to use a slightly different expression for the relative entropy. For any $\P \in M_+$ we let $\P^a$ denote its absolutely continuous part with respect to $\Q$. One then has
\begin{equation} \label{eq_rel_ent_alt}
H(\Q \mid \P) = \E_\Q\left[- \log \frac{d\P^a}{d\Q} \right].
\end{equation}
To see this, suppose first $d\P^a / d\Q$ is strictly positive $\Q$-almost surely. Then $\Q \ll \P^a \ll \P$ and $d\Q/d\P = d\Q/d\P^a = 1/(d\P^a / d\Q)$ $\Q$-almost surely, so \eqref{eq_rel_ent_alt} holds. If instead $d\P^a / d\Q$ is zero with positive $\Q$-probability, then the right-hand side of \eqref{eq_rel_ent_alt} is infinite (the expectation is always well-defined by Jensen's inequality), and so is the relative entropy because $\Q \not\ll \P$. Thus \eqref{eq_rel_ent_alt} holds in this case, too.

\begin{remark}
In the above, $\P$ can be any element of $M_+$. We do not adjust the definition of relative entropy for the possibility that its total mass is not one.
\end{remark}

We now make the trivial observation that 
\begin{equation} \label{eq_duality_1}
\E_\Q\left[ X \frac{d\P^a}{d\Q}\right] \le 1 \text{ for all $X \in \Ecal$ and $\P \in \Peff$.}
\end{equation}
Indeed, the left-hand side equals $\E_{\P^a}[X]$ which is dominated by $\E_\P[X]$, hence by one. This leads to the following theorem, which expresses the duality between the null hypothesis and its set of e-variables and justifies calling $\P^*$ the RIPr. We first treat the notationally simpler and statistically more relevant case where $\Q \ll \Pcal$ and, hence, the numeraire $X^*$ is finite. The general case is in Theorem~\ref{T_duality_general}.

\begin{theorem} \label{T_duality}
Assume $\Q \ll \Pcal$ and let $\P^*$ be an element of $\Peff$ equivalent to $\Q$. The following statements are equivalent, where as above we write $\P^a$ for the absolutely continuous part of $\P$ with respect to $\Q$ (and hence $\P^*$).
\begin{enumerate}
\item\label{T_duality_1} $\P^*$ is the \textnormal{RIPr}.
\item\label{T_duality_2} $\displaystyle \E_\Q\left[ \frac{d\P^a}{d\P^*} \right] \le 1 \text{ for all } \P \in \Peff$.
\item\label{T_duality_3} $\displaystyle \E_\Q\left[ \log \frac{d\P^a}{d\P^*} \right] \le 0 \text{ for all } \P \in \Peff.$
\end{enumerate}
Letting $X^*$ be the numeraire and $\P^*$ the \textnormal{RIPr}, one has the strong duality relation
\begin{equation} \label{eq_strong_duality}
\E_\Q[\log X^*] = \sup_{X \in \Ecal} \E_\Q[\log X] = \inf_{\P \in \Peff} H(\Q \mid \P) = H(\Q \mid \P^*),
\end{equation}
where some, and then all, of these quantities may be $+\infty$. Here $\E_\Q[\log X]$ is understood as $-\infty$ whenever $\E_\Q[(\log X)^-] = \infty$.  
\end{theorem}

Property~\ref{T_duality_2} satisfied by the RIPr can be viewed as dual to the numeraire property and generalizes Property~2 in \citet[Theorem 4.3]{li1999estimation}. The equivalence of \ref{T_duality_2} and \ref{T_duality_3} is analogous to the equivalence between the numeraire property and log-optimality of the numeraire established in Proposition~\ref{P_log_optimality}; we elaborate on this analogy in Remark~\ref{R_duality_analogy} below. Finally, \eqref{eq_strong_duality} states that strong duality holds between the problem of maximizing the expected logarithm over e-variables and the problem of minimizing relative entropy over the effective null. Note that $\E_\Q[\log X^*]$ is well-defined and nonnegative, possibly infinite, thanks to the numeraire property and the fact that $\log X^* \ge - 1/X^* + 1$. In the spirit of Remark~\ref{R_benchmark_numeraire}, we discuss a `benchmarked' version of \eqref{eq_strong_duality} in Remark~\ref{R_benchmark_duality} below.

The minimum relative entropy property of $\P^*$ justifies calling it the RIPr of $\Q$ on $\Peff$ as in Definition~\ref{D_RIPr}. In the terminology of \cite{harremoes2023universal}, property \ref{T_duality_3} states that the \emph{description gain} of switching from $\P^*$ to any other element of $\Peff$ is nonpositive, and this provides another justification for calling $\P^*$ the RIPr. We elaborate on the connection to their work in Remark~\ref{R_description_gain} below.
Existing literature on the RIPr assumes the existence of a reference measure along with much stronger versions of the condition that $\Q \ll \Pcal$. As we show in Section~\ref{S_verification}, in such settings our $\P^*$ coincides with existing notions of the RIPr of $\Q$ on $\Pcal$; see Theorem~\ref{T_ripr_mu} and its corollaries. However, our measure $\P^*$ exists in complete generality and satisfies properties normally associated with the RIPr; see Theorem~\ref{T_duality_general} below for the fully general case. This leads us to Definition~\ref{D_RIPr}, which we argue is the natural definition of the RIPr in the absence of any assumptions.

\begin{proof}[Proof of Theorem~\ref{T_duality}]
\ref{T_duality_1} $\Leftrightarrow$ \ref{T_duality_2}:
Suppose $\P^*$ is the RIPr. Because $\P^*$ and $\Q$ are equivalent we have $d\P^a/d\P^* = (d\P^a/d\Q) / (d\P^* / d\Q) = X^* d\P^a/d\Q$ up to $\Q$-nullsets for any $\P \in \Peff$. Thus we obtain \ref{T_duality_2} by taking $X = X^*$ in \eqref{eq_duality_1}. This proves the forward direction. To prove the converse, we show that the RIPr is the unique element of $\Peff$ that is equivalent to $\Q$ and satisfies \ref{T_duality_2}. This is argued exactly as uniqueness of the numeraire: if $\P_1^*, \P_2^* \in \Peff$ are both equivalent to $\Q$ and satisfy $ \E_\Q[ {d\P^a}/{d\P_i^*} ] \le 1$ for all $\P \in \Peff$, then with $Y = d\P^*_2/d\P^*_1$ we have $1 \le 1/\E_\Q[Y] \le \E_\Q[1/Y] \le 1$. Thus $Y$ is $\Q$-almost surely equal to one and $\P^*_1 = \P^*_2$.

\ref{T_duality_2} $\Leftrightarrow$ \ref{T_duality_3}: The forward implication follows from Jensen's inequality, and the converse uses the same argument as the proof of Proposition~\ref{P_log_optimality}. Specifically, $\P(t) = (1-t)\P^* + t \P$ is in $\Peff$ for all $t \in (0,1)$, so $\E_\Q[t^{-1}\log( d\P(t)^a/d\P^*)] \le 0$. Sending $t$ to zero yields \ref{T_duality_2}.

Finally, we argue \eqref{eq_strong_duality} assuming that $\P^*$ is the RIPr. Jensen's inequality and \eqref{eq_duality_1} give
\begin{equation} \label{duality_proof_1}
\E_\Q\left[ \log\left( X \frac{d\P^a}{d\Q} \right) \right] \le 0
\end{equation}
for all $X \in \Ecal$ and $\P \in \Peff$. This implies the `weak duality' inequality
\begin{equation} \label{duality_proof_2}
\E_\Q\left[ \log X \right] \le \E_\Q\left[ - \log \frac{d\P^a}{d\Q} \right] = H(\Q \mid \P).
\end{equation}
Indeed, if $(\log X)^-$ has infinite expectation, \eqref{duality_proof_2} is trivially true. Otherwise we first replace $X$ by $X \wedge n$, in which case \eqref{duality_proof_1} can be rearranged to \eqref{duality_proof_2} with $X \wedge n$ in place of $X$. Sending $n$ to infinity and using monotone convergence then gives \eqref{duality_proof_2}. Since equality is achieved by $X=X^*$ and $\P=\P^*$, we deduce \eqref{eq_strong_duality}.
\end{proof}

\begin{remark} \label{R_duality_analogy}
The equivalence of \ref{T_duality_2} and \ref{T_duality_3} in Theorem~\ref{T_duality}, and the equivalence stated in Proposition~\ref{P_log_optimality}, have commonalities that can be understood by considering two optimization problems. The first is to maximize $f(X) = \E_\Q[\log(X/X^*)]$ over $X \in \Ecal$. The second is to minimize $g(\P) = \E_\Q[-\log(d\P^a/d\P^*)]$ over $\P \in \Peff$. The numeraire property can be interpreted as the first-order condition that the Gateaux (i.e.\ directional) derivative $f'(X^*; X-X^*) = \lim_{t \downarrow 0} t^{-1}(f(X^* + t(X^*-X)) - f(X^*))$ in any feasible direction is nonpositive at $X^*$. Thanks to concavity, this is necessary and sufficient for optimality. Analogously, property \ref{T_duality_2} of Theorem~\ref{T_duality} states that the Gateaux derivative $g'(\P^*; \P-\P^*) = \lim_{t \downarrow 0} t^{-1}(g(\P^* + t(\P^*-\P)) - g(\P^*))$ in any feasible direction is nonnegative at $\P^*$. This is again necessary and sufficient for optimality, expressed by \ref{T_duality_3}. The first-order condition for $f$ states that $d\P^* = (1/X^*)d\Q$ yields an element of $\Peff$; and the first-order condition for $g(\P)$ states that $X^* = d\Q/d\P^*$ is an element of $\Ecal$. Furthermore, the two objective functions are related by $f(X) \le 0 \le g(\P)$, with equality precisely at the optimal solutions $X^*$ and $\P^*$. We refer to \cite{MR1727362} for more information on convex optimization in infinite dimension; see in particular Chapter~II therein.
\end{remark}

\begin{remark}\label{R_benchmark_duality}
Following Remark~\ref{R_benchmark_numeraire}, let us consider a benchmark $X_\text{b}$. In the presence of a benchmark, the strong duality relation \eqref{eq_strong_duality} takes the form
\[
\E_\Q\left[\log\frac{X^*}{X_\text{b}}\right]
= \sup_{X \in \Ecal} \E_\Q\left[\log\frac{X}{X_\text{b}}\right]
= \inf_{\P \in \Pcal_\text{eff}} \E_\Q\left[-\log\left( \frac{d\P^a}{d\Q} X_\text{b} \right)\right]
= \E_\Q\left[-\log\left( \frac{d\P^*}{d\Q} X_\text{b} \right)\right],
\]
where now all four quantities are finite. The optimizers $X^*$ and $\P^*$ do not depend on the choice of benchmark. This `benchmarked' version of strong duality is proved by replacing $X$ by $X/X_\text{b}$ and $d\P^a/d\Q$ by $X_\text{b}(d\P^a/d\Q)$ in \eqref{duality_proof_1} and the subsequent manipulations.
\end{remark}

\begin{remark} \label{R_description_gain}
In the spirit of \cite{harremoes2023universal}, one can define the \emph{maximum description gain} of switching from $\P \in \Peff$ to some other element of $\Peff$ by
\[
H(\Q \mid \P \leadsto \Peff) = \sup_{\P' \in \Peff} \E_\Q\left[ \log \frac{d\P'^a}{d\P} \right]
\]
if $\Q$ is absolutely continuous with respect to $\P$, and $+\infty$ otherwise. Thanks to condition \ref{T_duality_3} of Theorem~\ref{T_duality}, assuming that $\Q \ll \Pcal$, the RIPr is characterized as the unique element of $\Peff$ equivalent to $\Q$ such that $H(\Q \mid \P^* \leadsto \Peff) = 0$, and this is the smallest possible value that the maximum description gain can take.
\end{remark}

\begin{remark}
The strong duality in \eqref{eq_strong_duality} implies that there exists an e-variable with positive expected logarithm under $\Q$ if and only if $\Q \notin \Peff$. Therefore, thanks to Corollary~\ref{cor:conv-hull} below, Theorem~\ref{T_duality} extends a result by~\cite{zhang2023existence} showing that if $\Pcal$ is finite, then such an e-variable exists if and only if $\Q$ is not in the convex hull of $\Pcal$.
\end{remark}

\begin{remark}
The RIPr $\P^*$ is not a probability measure in general. Equivalently, it may happen that $\E_\Q[1/X^*] < 1$. We will see this phenomenon in Sections~\ref{S_param_null} and~\ref{S_symmetric_laws}, and other examples are given by \cite{harremoes2023universal}. It is also worth noting that there may exist elements $\widetilde\P \in \Peff$ that are different from $\P^*$, but whose absolutely continuous part $\widetilde\P^a$ is equal to $\P^*$. We will see this in Section~\ref{S_symmetric_laws} with $\widetilde\P$ being the symmetrization of $\Q$. Other examples are obtained from Example~\ref{ex_simple_null} below. 
\end{remark}

\begin{example}[Example~\ref{Ex:240225} continued] \label{ex_simple_null}
Consider a simple null hypothesis $\Pcal = \{\P_0\}$ with $\Q \ll \P_0$ so that the numeraire $X^*$ is finite. In this case $\P^* = \P_0^a$ and $X^* = 1/(d\P^*/d\Q) = d\Q/d\P_0$ up to $\Q$-nullsets. Here one can easily check that $\Peff$ consists of all sub-probability measures $\P$ such that $\P(A) \le \P_0(A)$ for all $A$. Indeed, any such $\P$ must belong to $\Peff$ since $\E_\P[X] \le \E_{\P_0}[X] \le 1$ for all e-variables $X$. Conversely, if $\P$ is not of this form, then $\P(A) > \P_0(A)$ for some $A$, and $X = \bm1_A / \P_0(A)$ is an e-variable such that $\E_\P[X] > 1$, showing that $\P \notin \Peff$. If $\P_0(A)=0$ we set $X = \infty \bm1_A$ in the previous sentence.
\end{example}

We end this section with the general version of Theorem~\ref{T_duality}. The proof reduces the general case to the one treated in Theorem~\ref{T_duality} by introducing the conditional probability measure $\Q^*(A) = \Q(A \mid X^* < \infty)$ if the quantity $\lambda^* = \Q(X^* < \infty)$ is positive, where $X^*$ is the numeraire. The fully degenerate case $\lambda^*=0$ is treated separately. It is evident from Definition~\ref{D_RIPr} that the RIPr $\P^*$ is absolutely continuous with respect to $\Q$, and we know from Lemma~\ref{L_P_star} that it belongs to $\Peff$. As shown below, one actually has more precise information, namely that
\begin{equation} \label{eq_scaled_bipolar}
\text{$\P^*$ is equivalent to $\Q$ on $\{X^* < \infty\}$ and belongs to $\lambda^* \cdot \Peff$},
\end{equation}
where we write $\lambda^* \cdot \Peff$ for the set $\{\lambda^* \P \colon \P \in \Peff\} \subset \Peff$. We continue to write $\P^a$ for the absolutely continuous part of $\P$ with respect to $\Q$. When applying this theorem it is useful to know that the event $\{X^* < \infty\}$ can be determined without first computing the numeraire; see Remark~\ref{R_num_finite} below.

\begin{theorem} \label{T_duality_general}
Let $X^*$ be the numeraire e-variable and define $\lambda^* = \Q(X^* < \infty)$. The \textnormal{RIPr} $\P^*$ satisfies \eqref{eq_scaled_bipolar}. Furthermore, for any sub-probability measure $\P^*$ that is absolutely continuous with respect to $\Q$ and satisfies \eqref{eq_scaled_bipolar}, the following statements are equivalent.
\begin{enumerate}
\item\label{T_duality_general_1} $\P^*$ is the \textnormal{RIPr}.
\item\label{T_duality_general_2} $\displaystyle \E_\Q\left[ \frac{d\P^a}{d\P^*} \bm1_{\{X^* < \infty\}} \right] \le 1 \text{ for all } \P \in \Peff.$
\item\label{T_duality_general_3} $\displaystyle \E_\Q\left[ \log \left( \lambda^* \frac{d\P^a}{d\P^*} \right) \bm1_{\{X^* < \infty\}} \right] \le 0 \text{ for all } \P \in \Peff.$
\end{enumerate}
If any of these hold, one has the strong duality relation \eqref{eq_strong_duality}. Finally,  suppose $\lambda^* \in (0,1)$, let $\P^*$ be the \textnormal{RIPr}, and define $\Q^* = \Q( \cdot \mid X^* < \infty)$. Then the numeraire and \textnormal{RIPr} associated with $\Q^*$ are $X^*$ and $\P^{**} = \P^*/\lambda^*$, and one has the `conditional' strong duality
\begin{equation} \label{eq_strong_duality_conditional}
\E_{\Q^*}[\log X^*] = \sup_{X \in \Ecal} \E_{\Q^*}[\log X] = \inf_{\P \in \Peff} H(\Q^* \mid \P) = H(\Q^* \mid \P^{**}).
\end{equation}
\end{theorem}

\begin{proof}
We first let $\P^*$ be the RIPr and check \eqref{eq_scaled_bipolar}. Equivalence with $\Q$ on $\{X^* < \infty\}$ is evident from Definition~\ref{D_RIPr}. To show that $\P^*$ belongs to $\lambda^* \cdot \Peff$ we refine the calculation in the proof of Lemma~\ref{L_P_star}. Consider any $X \in \Ecal$ and set $X' = X \bm1_{\{X^* < \infty\}} + \infty \bm1_{\{X^* = \infty\}}$. This still belongs to $\Ecal$ since it is $\P$-almost surely equal to $X$ for all $\P \in \Pcal$. Thanks to the convention \eqref{eq_arithmetic} we have $X'/X^* = \bm1_{\{X^* < \infty\}}X'/X^* + \bm1_{\{X^* = \infty\}}$ and hence
\[
\E_{\P^*}[X] = \E_{\P^*}[X'] = \E_\Q\left[ \bm1_{\{X^* < \infty\}} \frac{X'}{X^*} \right] = \E_\Q\left[ \frac{X'}{X^*} \right] - \Q(X^* = \infty) \le \Q(X^* < \infty),
\]
using the numeraire property in the last step. It follows that $\P^* \in \lambda^* \cdot \Peff$.

We now let $\P^*$ be any sub-probability measure that is absolutely continuous with respect to $\Q$ and satisfies \eqref{eq_scaled_bipolar}, and prove the claimed equivalences and the strong duality statement. First, in the fully degenerate case $\lambda^* = 0$, \eqref{eq_scaled_bipolar} just states that $\P^* = 0$. This is also the RIPr, so \ref{T_duality_general_1} is true. Moreover, \ref{T_duality_general_2} and \ref{T_duality_general_3} reduce to the statements that $0 \le 1$ and $0 \le 0$. Thus all three statements are true and and hence trivially equivalent. The strong duality \eqref{eq_strong_duality} holds because all four terms are infinite. Second, in the `fully non-degenerate' case $\lambda^* = 1$, the theorem simply reduces to Theorem~\ref{T_duality}, which we have already proved.

Consider now the case $\lambda^* \in (0,1)$ and define $\Q^* = \Q( \cdot \mid X^* < \infty)$. Taking $\Q^*$ as our new alternative hypothesis, $X^*$ is still the numeraire. Indeed, it is a $\Q^*$-almost surely strictly positive e-variable, and for any other e-variable $X$ we define $X'$ as above and get
\[
\E_{\Q^*}\left[\frac{X}{X^*}\right] = \frac{1}{\lambda^*} \E_\Q\left[\bm1_{\{X^*<\infty\}} \frac{X}{X^*}\right] = \frac{1}{\lambda^*} \left(\E_\Q\left[\frac{X'}{X^*}\right] - \Q(X^*=\infty)\right) \le 1.
\]
Since $\Q^*(X^* < \infty) = 1$ by construction, we have from Theorem~\ref{T_numeraire} that $\Q^* \ll \Pcal$. We also define $\P^{**} = \P^* / \lambda^*$ which, thanks to \eqref{eq_scaled_bipolar}, is an element of $\Peff$ equivalent to $\Q^*$. Now, the statement \ref{T_duality_general_1} that $\P^*$ is the RIPr of $\Q$ is equivalent to $\P^{**}$ being the RIPr of $\Q^*$. Furthermore, \ref{T_duality_general_2} and \ref{T_duality_general_3} can be equivalently expressed in terms of $\Q^*$ and $\P^{**}$ as
\[
\E_{\Q^*}\left[ \frac{d\P^a}{d\P^{**}} \right] \le 1 \text{ for all } \P \in \Peff
\quad \text{and} \quad
\E_{\Q^*}\left[ \log \frac{d\P^a}{d\P^{**}} \right] \le 0 \text{ for all } \P \in \Peff,
\]
respectively. 
We may thus simply apply Theorem~\ref{T_duality} with $\Q^*$ and $\P^{**}$ in place of $\Q$ and $\P^*$ to deduce that the three statements are indeed equivalent and that the `conditional' strong duality \eqref{eq_strong_duality_conditional} holds. The `unconditional' strong duality \eqref{eq_strong_duality} holds because all four terms are infinite.
\end{proof}

\begin{remark} \label{R_num_finite}
The event $\{X^* < \infty\}$ can be determined without first computing the numeraire. Let $\Q = \Q^r + \Q^s$ be the generalized Lebesgue decomposition of $\Q$ with respect to $\Pcal$ given by Lemma~\ref{L_gen_Lebesgue}, and let $N$ be the event given in the lemma. Since $N$ is $\Pcal$-negligible, $X = \infty \bm1_N$ is an e-variable, so the numeraire must be infinite on $N$, $\Q$-almost surely. On the other hand, because the event $N' = \{X^* = \infty\}$ is $\Pcal$-negligible, the lemma also yields $\Q(N' \setminus N) = 0$. Thus $N$ and $\{X^* = \infty\}$ are both $\Pcal$-negligible and coincide up to $\Q$-nullsets, so the former can be used in place of the latter in Theorem~\ref{T_duality_general}.
\end{remark}

Following up on Example~\ref{ex_simple_null}, the following corollary records the RIPr and numeraire for a point null, generalizing a well known result that the likelihood ratio is the optimal e-variable if the two distributions are equivalent; see for example~\cite{shafer2021testing}.

\begin{corollary}
Consider an arbitrary point null $\Pcal = \{\P_0\}$. The \textnormal{RIPr} is the absolutely continuous part $\P^* = \P_0^a$ with respect to $\Q$, and the numeraire is
\begin{equation} \label{eq_num_general_point_null}
X^* = \infty\bm1_N + \frac{d\Q^*}{d\P_0}\bm1_{N^c},
\end{equation}
where $N$ is a set such that $\P_0(N) = 0$ and $\Q$ restricted to $N^c$ is absolutely continuous with respect to $\P_0$, and $\Q^* = \Q(\cdot \mid N^c)$ whenever $\Q(N^c) > 0$.
\end{corollary}

\begin{proof}
Remark~\ref{R_num_finite} implies that $N$ equals the event where the numeraire is infinite up to $\Q$-nullsets, so $\Q^*$ defined in the corollary coincides with $\Q^*$ defined in Theorem~\ref{T_duality_general}. Moreover, thanks to Example~\ref{ex_simple_null}, the numeraire for $\Q^*$ is $d\Q^*/d\P_0$ up to $\Q^*$-nullsets. From these facts and Theorem~\ref{T_duality_general}, we conclude that the numeraire for $\Q$ is given by \eqref{eq_num_general_point_null}. Finally, since $1/X^* = d\P_0^a / d\Q$ up to $\Q$-nullsets, the \textnormal{RIPr} is $\P^* = \P_0^a$ as claimed.
\end{proof}

\section{Finding the numeraire and RIPr} \label{S_verification}

So far we have studied the abstract existence, uniqueness, and duality theory of the numeraire and RIPr, but we have not said much about how to actually compute them. In this section we develop tools for doing so, which we then apply to several examples in Section~\ref{S_examples}. 

We focus on the statistically most relevant situation where $\Q \ll \Pcal$. In this case the numeraire $X^*$, and indeed every e-variable, is $\Q$-almost surely finite and the RIPr $\P^*$ is equivalent to~$\Q$. The general case can be reduced to this case by replacing $\Q$ with the conditional measure $\Q^*$ in Theorem~\ref{T_duality_general} (except in the fully degenerate situation where the numeraire is infinity and the RIPr zero.)

Proposition~\ref{P_log_optimality} suggests that $X^*$ can be found by solving an optimization problem. Unfortunately this may not be straightforward because the optimization is over the set of \emph{all} e-variables, which can be difficult to characterize. For the same reason, even if a candidate $X^*$ has been found it may not be clear how to check that it does, in fact, satisfy the numeraire property. Alternatively, the strong duality of Theorem~\ref{T_duality} suggests minimizing relative entropy to find $\P^*$. However, the optimization is again over a set that can be difficult to describe explicitly, in this case the effective null $\Peff$. Moreover, checking that a candidate $\P^*$ is optimal involves comparing it to \emph{all} elements of $\Peff$.

We first present a verification theorem which simplifies the task of checking that candidates $X^*$ and $\P^*$ are in fact optimal. This result strengthens Corollary~2 of \cite{grunwald2020safe}. Despite the simple proof, the result turns out to be remarkably useful.

\begin{theorem} \label{T_verification}
Assume that $\Q \ll \Pcal$. Let $X^*$ be a $\Q$-almost surely strictly positive e-variable and let $\P^*$ be the measure given by $d\P^*/d\Q = 1/X^*$. Then $X^*$ is a numeraire if and only if $\P^*$ belongs to the effective null $\Peff$. Thus the numeraire is the only e-variable that is also a likelihood ratio between $\Q$ and some equivalent element of $\Peff$. Finally, $\E_{\P^*}[X^*] = 1$ and $\P^*$ is maximal in the sense that no other $\P \in \Peff$ equivalent to $\Q$ can satisfy $\P(A) \ge \P^*(A)$ for all $A$ with strict inequality for some $A$.
\end{theorem}

\begin{proof}
Recall that all e-variables are $\Q$-almost surely finite thanks to Theorem~\ref{T_numeraire} and the assumption that $\Q \ll \Pcal$; in particular, $\P^*$ is equivalent to $\Q$. Now, if $X^*$ is a numeraire, then Lemma~\ref{L_P_star} yields that  $\P^* \in \Peff$. For the converse, the definition of $\P^*$ and the fact that it belongs to $\Peff$ yield
$\E_\Q[X/X^*] = \E_{\P^*}[X] \le 1$ for any e-variable $X$. This is the numeraire property. Finally, the definition of $\P^*$ immediately gives $\E_{\P^*}[X^*] = 1$, and $\P^*$ must be maximal because if a `dominating' equivalent $\P \in \Peff$ existed we would get the contradiction $1 = \E_{\P^*}[X^*] < \E_\P[X^*] \le 1$.
\end{proof}

\subsection{Generated null hypotheses}

The following corollary of Theorem~\ref{T_verification} applies in settings where the null hypothesis is `generated', in the sense of \eqref{T_verification_1} below, by a subset $\Ecal_0$ of e-variables. Null hypotheses of this kind were studied by \cite{MR4699555} in a sequential setting. The corollary shows that once a candidate $X^*$ has been found, it is enough to check the numeraire property for the generating family $\Ecal_0$, provided that it holds with equality for the e-variable which is identically one. The latter condition means that the RIPr $\P^*$ is a probability measure.

\begin{corollary} \label{C:240207}
Assume that $\Q \ll \Pcal$. Let $\Ecal_0$ be a family of $[0,\infty]$-valued random variables that generates the null hypothesis in the sense that
\begin{equation} \label{T_verification_1}
\Pcal = \{\P \in M_1 \colon \E_\P[X] \le 1 \text{ for all } X \in \Ecal_0\}.
\end{equation}
Let $X^*$ be a $\Q$-almost surely strictly positive e-variable such that
\begin{equation} \label{T_verification_2}
\E_\Q\left[ \frac{1}{X^*} \right] = 1 \text{ and } \E_\Q\left[ \frac{X}{X^*} \right] \le 1 \text{ for all } X \in \Ecal_0.
\end{equation}
Then $X^*$ is the numeraire and the \textnormal{RIPr} belongs to $\Pcal$.
\end{corollary}

\begin{proof}
Let $\P^*$ be given by $d\P^*/d\Q = 1/X^*$. Then
\eqref{T_verification_2} says that $\P^*$ is a probability measure such that $\E_{\P^*}[X] \le 1$ for all $X \in \Ecal_0$. Thus by \eqref{T_verification_1}, $\P^*$ belongs to $\Pcal$, hence to $\Peff$. The result now follows from Theorem~\ref{T_verification}.
\end{proof}

Thanks to Lemma~\ref{lem:lifting}, Corollary~\ref{C:240207} also holds if one has `$\supset$' instead of `$=$' in \eqref{T_verification_1}. This is intuitive: if one enlarges the generating set $\Ecal_0$ beyond what is necessary to specify $\Pcal$, then the right hand side of~\eqref{T_verification_1} can become smaller than $\Pcal$. In that case, \eqref{T_verification_2} only gets harder to satisfy, and if we find an e-variable $X^*$ satisfying it, it surely must still be the numeraire.

One may wonder if the equality in \eqref{T_verification_2} could be replaced by an inequality; this would amount to a significant strengthening of the corollary. The following example shows that this is unfortunately not possible.

\begin{example}
Consider the single coin toss space $\Omega = \{0,1\}$ and the random variable $X_0$ given by $X_0(0) = 2$ and $X_0(1) = 1/2$. Let the singleton set $\Ecal_0 = \{X_0\}$ generate the null hypothesis $\Pcal$, which then consists of all probability measures $\P$ with $\P(0) \le 1/3$. Next, let the alternative $\Q$ be given by $\Q(0)=1$, $\Q(1)=0$. We then have $\E_\Q[1/X_0] = 1/2 < 1$ and, trivially, $\E_\Q[X/X_0] \le 1$ for $X \in \Ecal_0$. If Corollary~\ref{C:240207} were to remain true with an inequality instead of an equality in \eqref{T_verification_2}, we would conclude that $X_0$ is the numeraire. However, this is not so. Indeed, for the e-variable $X^*$ given by $X^*(0) = 3$, $X^*(1) = 0$, we have $\E_\Q[X^*/X_0] = 3/2 > 1$. As suggested by the notation, $X^*$ is in fact the numeraire in this example. To see this, just note that every e-variable $X$ must satisfy $X(0) \le 3$ since $\Pcal$ contains the probability measure $\P$ with $\P(0) = 1/3$, and hence $\E_\Q[X/X^*] = X(0)/3 \le 1$. It is also worth noting that $\E_\Q[1/X^*] = 1/3 < 1$, which means that we cannot use Corollary~\ref{C:240207} to detect the numeraire property of $X^*$. The corollary is not a universal tool, but nonetheless a useful one as we will see in Section~\ref{S_examples}.
\end{example}

\subsection{When a reference measure exists} \label{S_ref_measure}

We now consider the dual approach of finding $\P^*$ directly, either by minimizing relative entropy or, if a candidate is available, checking that \ref{T_duality_2} or \ref{T_duality_3} of Theorem~\ref{T_duality} are satisfied. The difficulty is that the domain is all of $\Peff$, which can be hard to describe explicitly. Matters would become simpler if one could instead work directly with $\Pcal$, but this is not possible in general. Here is one example of what can go wrong.

\begin{example} \label{ex_P_bipolar}
This example shows that Theorem~\ref{T_duality}\ref{T_duality_2} may no longer characterize the \textnormal{RIPr} $\P^*$ uniquely if $\Peff$ is replaced by $\Pcal$. Let $\Omega = [0,1]$ with its Borel $\sigma$-algebra, and let $\Pcal = \{\delta_\omega \colon \omega \in [0,1]\}$ be the set of all Dirac point masses. In this case the set of e-variables is trivial in the sense that a random variable $X$ is an e-variable if and only if $0 \le X(\omega) \le 1$ for all $\omega \in [0,1]$. It follows that $\Peff$ contains \emph{all} sub-probabilities on $\Omega$. Now, let the alternative hypothesis $\Q$ be the uniform distribution on the unit interval. We then have the extreme situation that $\Pcal$ does not contain \emph{any} distribution absolutely continuous with respect to $\Q$, whereas $\Peff$ contains \emph{all} such distributions. In particular, Theorem~\ref{T_duality}\ref{T_duality_2} would be vacuously true for any candidate $\P^*$ if $\Peff$ were replaced by $\Pcal$. In contrast, as currently stated, Theorem~\ref{T_duality}\ref{T_duality_2} will only hold for $\P^* = \Q$, which is the correct answer.
\end{example}

The situation improves in the presence of a reference measure $\mu$ with respect to which $\Q$ and all $\P \in \Pcal$ are absolutely continuous. In this setting, if $\P_n$ and $\P$ are measures with densities $p_n = d\P_n/d\mu$ and $p = d\P/d\mu$ we say that $\P_n$ \emph{$\mu$-converges} to $\P$ if $p_n$ converges to $p$ in probability under $\mu$. A set $C$ of absolutely continuous measures is called \emph{$\mu$-closed} if it is closed with respect to $\mu$-convergence. It is called \emph{solid} if it contains every (nonnegative) measure $\P'$ that is dominated by some $\P \in C$ in the sense that $\P'(A) \le \P(A)$ for all $A$. The following result characterizes $\Peff$ in terms of these concepts.

\begin{lemma} \label{L_P_bipolar_char_ref_measure}
Suppose all $\P \in \Pcal$ are absolutely continuous with respect to a probability measure $\mu$. Then all elements of $\Peff$ are also absolutely continuous with respect to $\mu$, and $\Peff$ is the smallest $\mu$-closed convex solid set that contains $\Pcal$.
\end{lemma}

\begin{proof}
We identify $\Pcal$ with its set of densities $C = \{p \in L^0_+(\mu) \colon p = d\P/d\mu,\ \P \in \Pcal\}$. The polar of $C$ in $L^0_+(\mu)$ is $C^\circ = \{f \in L^0_+(\mu) \colon \E_\mu[pf] \le 1 \text{ for all } p \in C\}$, which is the set of (equivalence classes of) random variables that are $\mu$-almost surely equal to a finite e-variable. We claim that all elements of $\Peff$ are absolutely continuous with respect to $\mu$, and that its set of densities is precisely the bipolar $C^{\circ\circ} = \{p \in L^0_+(\mu) \colon \E_\mu[pf] \le 1 \text{ for all } f \in C^\circ\}$. To check absolute continuity, let $\P \in \Peff$ and consider any event $A$ such that $\mu(A)=0$. Then $\infty \bm1_A$ is an e-variable, hence $\E_\P[\infty \bm1_A] \le 1$, and thus $\P(A) = 0$. This shows absolute continuity and lets us identify $\Peff$ with its set of densities, namely, the set of $p \in L^0_+(\mu)$ such that $\E_\mu[p X] \le 1$ for all finite e-variables $X$. Here we may restrict to finite e-variables because $\E_\mu[p X] \le 1$ if and only if $\E_\mu[p (X \wedge n)] \le 1$ for all $n \in \N$. The finite e-variables are, up to $\mu$-nullsets, precisely the elements of $C^\circ$, so we deduce that $\Peff$ is indeed identified with $C^{\circ\circ}$. Finally, the 
bipolar theorem of \citet[Theorem~1.3]{Brannath} states that $C^{\circ\circ}$ is the smallest $\mu$-closed convex solid set that contains $C$. This completes the proof.
\end{proof}

\begin{corollary}\label{cor:conv-hull}
If $\Pcal$ is finite, then $\Peff$ consists of all sub-probabilities that are set-wise dominated by an element of $\conv(\Pcal)$. In particular, $\Peff \cap M_1 = \conv(\Pcal)$.
\end{corollary}
\begin{proof}
For finite $\Pcal = \{\P_1,\ldots,\P_n\}$ we may take $\mu = (\P_1 + \cdots + \P_n)/n$ as reference measure. In this case $\conv(\Pcal)$ is already $\mu$-closed and the form of $\Peff$ follows from Lemma~\ref{L_P_bipolar_char_ref_measure}.
\end{proof}

Lemma~\ref{L_P_bipolar_char_ref_measure} is powerful because it characterizes \emph{all} elements of $\Peff$. The following theorem makes use of this to simplify the conditions of Theorem~\ref{T_duality}. Lemma~\ref{L_P_bipolar_char_ref_measure} also provides a way of checking that a particular measure belongs to $\Peff$. For example, it shows that limits in probability of (densities of) measures in the convex hull of $\Pcal$ must belong to $\Peff$. The latter is however also easy to check directly from the definition of $\Peff$ and Fatou's lemma, and does not require the bipolar theorem used in the proof of Lemma~\ref{L_P_bipolar_char_ref_measure}.

If a reference measure $\mu$ is given, we identify any $\P \in \Pcal$ with its density $p = d\P/d\mu$. Thanks to Lemma~\ref{L_P_bipolar_char_ref_measure}, we may similarly identify $\Peff$ with its set of densities. For the sake of brevity we abuse notation and write, for example, $p \in \Pcal$ even though, strictly speaking, $\Pcal$ is a set of measures, not densities. Next, given $\Q$ with density $q$, the absolutely continuous part of any $p \in \Peff$ has density $p^a = p\bm1_{\{q>0\}}$, and this is $\Q$-almost surely equal to $p$ itself. In view of these observations, the conditions \ref{T_duality_2} and \ref{T_duality_3} of Theorem~\ref{T_duality} characterizing the RIPr state, respectively, that $\E_\Q[p/p^*] \le 1$ and $\E_\Q[\log(p/p^*)] \le 0$, for all $p \in \Peff$. Here $p^* = d\P^*/d\mu$. The following theorem simplifies these conditions by showing that, essentially, it is enough to check them for elements of $\Pcal$ rather than $\Peff$.

\begin{theorem} \label{T_ripr_mu}
Assume that $\Q \ll \Pcal$ and that $\Q$ and all elements of $\Pcal$ are absolutely continuous with respect to a probability measure $\mu$. Let $p^* \in \Peff$ be equivalent to $\Q$. Then each of the conditions \ref{T_duality_1}, \ref{T_duality_2}, and \ref{T_duality_3} of Theorem~\ref{T_duality} is equivalent to
\begin{equation}\label{eq_dual_numeraire_propoerty_mu}
\E_\Q\left[ \frac{p}{p^*} \right] \le 1 \text{ for all $p \in \Pcal$,}
\end{equation}
as well as to
\begin{equation} \label{eq_description_gain_mu}
\E_\Q\left[ \log \frac{(1-t)p^* + t p}{p^*} \right] \le 0 \text{ for all $t \in (0,1)$ and $p \in \Pcal$.}
\end{equation}
In either case, $p^*$ is (the density of) the \textnormal{RIPr}.
\end{theorem}

\begin{remark}
The reason for the asymmetry between Theorem~\ref{T_duality}\ref{T_duality_3} and \eqref{eq_description_gain_mu} is that the latter only involves $p$ from $\Pcal$, not $\Peff$. Since $\Pcal$ may not contain $p^*$, nor be convex, we need to explicitly include convex combinations with $p^*$ in the numerator on the left-hand side of \eqref{eq_description_gain_mu}. Let us also emphasize that $t=1$ is excluded in \eqref{eq_description_gain_mu}. This sometimes makes the condition easier to check, for instance in the proof of Corollary~\ref{C_description_gain} below.
\end{remark}

\begin{proof}[Proof of Theorem~\ref{T_ripr_mu}]
We know that \ref{T_duality_1}, \ref{T_duality_2}, and \ref{T_duality_3} of Theorem~\ref{T_duality} are equivalent for elements of $\Peff$ that are equivalent to $\Q$. To prove the theorem it is therefore enough to show that that Theorem~\ref{T_duality}\ref{T_duality_2}, \eqref{eq_dual_numeraire_propoerty_mu}, and \eqref{eq_description_gain_mu} are equivalent.

We first show that Theorem~\ref{T_duality}\ref{T_duality_2} is equivalent to \eqref{eq_dual_numeraire_propoerty_mu}. The former is clearly a stronger condition because it states that \eqref{eq_dual_numeraire_propoerty_mu} holds with the larger set $\Peff$ in place of $\Pcal$. For the converse, suppose \eqref{eq_dual_numeraire_propoerty_mu} holds. Let $\Pcal'$ be the set of all sub-probability densities dominated by some element of the convex hull of $\Pcal$, and let $\Pcal''$ be the closure in $\mu$-probability of $\Pcal'$. It is clear that \eqref{eq_dual_numeraire_propoerty_mu} holds for all $p \in \Pcal'$ and then, by Fatou's lemma, also for all $p \in \Pcal''$. We claim that $\Pcal''$ is a closed convex solid set that contains $\Pcal$; it must then also contain $\Peff$ since this is the \emph{smallest} such set by Lemma~\ref{L_P_bipolar_char_ref_measure}. From this we conclude that Theorem~\ref{T_duality}\ref{T_duality_2} holds. To see that $\Pcal''$ is closed, convex, and solid, note that by definition, every $p'' \in \Pcal''$ is the limit in probability of some $p_n' \in \Pcal'$, meaning that $p_n' \le p_n$ for some $p_n$ in the convex hull of $\Pcal$. From this description convexity follows directly, closedness is obtained from a diagonal argument, and solidity follows by noting that any $\widetilde p \le p''$ is the limit of $\widetilde p \wedge p_n'$, which is still dominated by $p_n$ and thus belongs to $\Pcal'$.

The equivalence of \eqref{eq_dual_numeraire_propoerty_mu} and \eqref{eq_description_gain_mu} is argued as the equivalence of \ref{T_duality_2} and \ref{T_duality_3} of Theorem~\ref{T_duality}, or the equivalence of the numeraire property and log-optimality in Proposition~\ref{P_log_optimality}. Specifically, the forward implication uses Jensen's inequality and \eqref{eq_dual_numeraire_propoerty_mu} to get $\E_\Q[\log(1-t + tp/p^*)] \le \log(1 - t + t \E_\Q[p/p^*]) \le 0$, whereas for the reverse implication one divides \eqref{eq_description_gain_mu} by $t$, which is then sent to zero to yield \eqref{eq_dual_numeraire_propoerty_mu}.
\end{proof}

Even with a reference measure, it is possible that the relative entropy $H(\Q \mid \P)$ is infinite for all $\P \in \Pcal$ but finite for the \textnormal{RIPr} $\P^*$, similarly to what happens in Example~\ref{ex_P_bipolar}. The following example illustrates this.

\begin{example} \label{ex_cauchy_mixture_of_gaussians}
Let $\Omega = \R$ with its Borel $\sigma$-algebra, let $\Pcal$ be the set of all finite mixtures of Gaussians, and let $\Q$ be the standard Cauchy distribution. We may take $\mu = \Q$ as the reference measure, for example. With this setup, $H(\Q \mid \P) = \infty$ for all $\P \in \Pcal$. On the other hand, the Cauchy density can be obtained as the limit in $\mu$-probability (even in the uniform norm) of elements of $\Pcal$, so it follows from Lemma~\ref{L_P_bipolar_char_ref_measure} that $\Q \in \Peff$. Thus the \textnormal{RIPr} is $\P^* = \Q$ and $H(\Q \mid \P^*) = 0$.
\end{example}

We continue to identify measures with their $\mu$-densities. Theorem~\ref{T_ripr_mu} can be used to recover the following result of \cite{harremoes2023universal}, who define the \emph{maximum description gain} of switching from a particular $p \in \Pcal$ to some other element of $\Pcal$ by
\[
H(\Q \mid p \leadsto \Pcal) = \sup_{p' \in \Pcal} \E_\Q\left[\log \frac{p'}{p}\right]
\]
if $\Q$ is absolutely continuous with respect to $p$, and $+\infty$ otherwise. Note that this differs from the quantity $H(\Q \mid \P \leadsto \Peff)$ in Remark~\ref{R_description_gain}, which involves $\Peff$ rather than $\Pcal$.

\begin{corollary} \label{C_description_gain}
Assume that $\Q \ll \Pcal$ and that all elements of $\Pcal$ are absolutely continuous with respect to a probability measure $\mu$.  
If $\Pcal$ is convex and $(p_n)_{n \in \N}$ is a sequence in $\Pcal$ with $H(\Q \mid p_n \leadsto \Pcal) \to 0$, then $\log p_n$ converges to $\log p^*$ in $L^1(\Q)$, where $p^*$ is the \textnormal{RIPr}. 
\end{corollary}

\begin{proof}
The key observation of \cite{harremoes2023universal} is that the description gain controls the $L^1(\Q)$ distance between the log-densities $\log p_n$, showing that they form a Cauchy sequence. We give a self-contained version of their argument. We claim that for any $t > 0$ one has
\begin{equation} \label{eq_log_t_ineq}
\frac{1}{9} \left(|\log t| \wedge |\log t|^2 \right)
\le \log\left(\frac{2 + t + t^{-1}}{4}\right)
= \log\left(\frac{1 + t^{-1}}{2}\right) + \log\left(\frac{t + 1}{2}\right). 
\end{equation}
To verify the inequality, note that it is equivalent to $f(x) \ge 0$ for $x > 0$, where $f(x) = \log((2+e^x+e^{-x})/4) - (x \wedge x^2)/9$. Since $f''(x) = 2/(2+e^x+e^{-x}) - (2/9)\bm1_{(0,1)}(x) > 0$ for $x \ne 1$, and since $f(0) = f'(0) = 0$, we deduce that $f(x) > 0$ for $x > 0$.

Setting $t = p_m/p_n$ in \eqref{eq_log_t_ineq}, taking expectation under $\Q$, and using that $\Pcal$ is convex and hence contains $(p_m+p_n)/2$, we get
\[
\frac{1}{9} \E_\Q\left[ \left|\log \frac{p_m}{p_n}\right| \wedge \left|\log \frac{p_m}{p_n}\right|^2 \right]
\le \E_\Q\left[ \log\left( \frac{p_m+p_n}{2p_m} \right) \right] + \E_\Q\left[ \log\left( \frac{p_m+p_n}{2p_n} \right) \right]
\le H_m + H_n,
\]
where we use the shorthand notation $H_n = H(\Q \mid p_n \leadsto \Pcal)$. From the general inequality $\E[ Y ] \le \E[ Y \wedge Y^2 ] + \sqrt{\E[ Y \wedge Y^2 ]}$ for nonnegative random variables $Y$ we then deduce
\[
\E_\Q\left[ |\log p_m - \log p_n| \right] \le 9(H_m + H_n) + 3 \sqrt{H_m + H_n}.
\]
This shows that $\log p_n$, $n \in \N$, is a Cauchy sequence in $L^1(\Q)$ and converges to some limit which we may write as $\log p^*$ for some $\Q$-almost surely positive random variable $p^*$. To ensure that $p^*$ is the density of a measure that is equivalent to $\Q$ we choose it to be zero on the $\Q$-nullset $\{q=0\}$, where $q = d\Q/d\mu$. This is easily achieved by replacing $p^*$ with $p^*\bm1_{\{q>0\}}$ if necessary. Now, convergence in $L^1$ implies convergence in probability, which is preserved under continuous transformations, so $p_n$ converges to $p^*$ in probability under $\Q$. Furthermore, for any $t \in (0,1)$ and $p \in \Pcal$ we have $\log \frac{(1-t)p_n + tp}{p_n} \ge \log(1-t) > -\infty$, so Fatou's lemma yields
\[
\E_\Q\left[\log \frac{(1-t)p^* + t p}{p^*}\right] \le \liminf_{n\to\infty} \E_\Q\left[\log \frac{(1-t)p_n + tp}{p_n}\right] \le \lim_{n\to\infty} H_n = 0 \text{ for all } p \in \Pcal.
\]
Theorem~\ref{T_ripr_mu} now shows that $p^*$ is the RIPr, provided we can argue that $p^*$ belongs to $\Peff$. This will use that $p^* = 0$ on $\{q=0\}$. Specifically, note that $\Peff$ contains $p_n$ and hence, being solid, also the smaller random variable $p_n \bm1_{\{q>0\}}$. Convergence in $\Q$-probability of $p_n$ to $p^*$ implies convergence in $\mu$-probability of $p_n \bm1_{\{q>0\}}$ to $p^*\bm1_{\{q > 0\}} = p^*$. Since $\Peff$ is closed in $\mu$-probability, it must contain $p^*$. This completes the proof.
\end{proof}

The maximum description gain is useful in situations where the relative entropy $H(\Q \mid p)$ is infinite for all $p \in \Pcal$. On the other hand, in many cases the relative entropies are finite and one recovers the original result of \cite{li1999estimation} on existence of the reverse information projection, as was also observed by \cite{harremoes2023universal}.

\begin{corollary} \label{C_min_rel_entropy}
Assume that $\Q \ll \Pcal$ and that all elements of $\Pcal$ are absolutely continuous with respect to a probability measure $\mu$.  
If $\Pcal$ is convex, $H(\Q \mid p) < \infty$ for some $p \in \Pcal$, and $(p_n)_{n \in \N}$ is an entropy minimizing sequence in $\Pcal$, meaning that $H(\Q \mid p_n) \to \inf_{p \in \Pcal} H(\Q \mid p)$, then $\log p_n$ converges to $\log p^*$ in $L^1(\Q)$, where $p^*$ is the \textnormal{RIPr}.  In particular, if some $p' \in \Pcal$ achieves the infimum, then $p'$ is the RIPr. Finally, $H(Q \mid p^*) = \inf_{p \in \Pcal} H(\Q \mid p)$.
\end{corollary}

\begin{proof}
Since the quantities involved are finite for $n$ large enough, $\lim_{n\to\infty} H(\Q \mid p_n \leadsto \Pcal) = \lim_{n\to\infty} H(\Q \mid p_n) - \inf_{p \in \Pcal} H(\Q \mid p) = 0$, and the first claim follows from Corollary~\ref{C_description_gain}.   The second claim follows by simply taking $p_n = p'$ for all $n$. The final claim follows because $H(\Q \mid p_n) \to H(\Q \mid p^*)$ thanks to the $L^1(Q)$ convergence of the log-densities.
\end{proof}

Corollaries~\ref{C_description_gain} and~\ref{C_min_rel_entropy} confirm that the definition of the RIPr used in this paper, Definition~\ref{D_RIPr}, is consistent with existing definitions in the literature. Specifically, \cite{li1999estimation} defines the RIPr as the density $p^*$ with the property that whenever $(p_n)_{n \in \N}$ is a sequence in $\Pcal$ such that $H(\Q \mid p_n) \to \inf_{p \in \Pcal} H(\Q \mid p)$, we have $\log p_n \to \log p^*$ in $L^1(\Q)$. Corollary~\ref{C_min_rel_entropy} states that this indeed yields the RIPr of Definition~\ref{D_RIPr}. Similarly, \cite{harremoes2023universal} define the RIPr as the density $p^*$ with the property that whenever $(p_n)_{n \in \N}$ is a sequence in $\Pcal$ such that $H(\Q \mid p_n \leadsto \Pcal) \to \inf_{p \in \Pcal} H(\Q \mid p \leadsto \Pcal)$ (which must be zero whenever it is finite), $p_n$ converges to $\P^*$ in a certain metric which is equivalent to $L^1(\Q)$-convergence of the log-densities. Thanks to Corollary~\ref{C_description_gain}, this again yields the RIPr of Definition~\ref{D_RIPr}.

\section{Examples} \label{S_examples}

We now turn to examples. In these examples we consider a single random observation $Z$ in a measurable space $\Zcal$. For concreteness we take $(\Omega, \Fcal)$ to be $\Zcal$ with its $\sigma$-algebra, and $Z$ the canonical random variable. We will always have $\Q \ll \Pcal$. We start with a simple parametric example, before moving on to more sophisticated nonparametric examples where there exists no reference measure, in order to display the power and generality of our theory. We end with another parametric example. Except for the first and the last example, past work like \cite{wasserman2020universal,grunwald2020safe,harremoes2023universal} cannot handle the situations discussed here.

\subsection{A parametric null} \label{S_param_null}

We start with a simple parametric example from \cite{harremoes2023universal} involving a real-valued observation $Z$ in $\Zcal = \R$, and then generalize our observations at the end of this subsection. The null hypothesis $\Pcal$ consists of two unit-variance Gaussians $\P_1$ and $\P_2$ with mean $+1$ and $-1$, and the alternative hypothesis $\Q$ is the standard Cauchy distribution. This example was considered by the aforementioned authors because $H(Q \mid P_j) = \infty$ for $j \in \{1,2\}$. 

We claim that the RIPr is $\P^* = (\P_1+\P_2)/2$. This probability measure is equivalent to $\Q$ and belongs to $\Peff$, which always contains the convex hull of $\Pcal$. Furthermore, writing $p_1,p_2,p^*,q$ for the Lebesgue densities, we have $p_2(z) = p_1(-z)$ and $q(z) = q(-z)$ and hence
\[
\E_\Q\left[ \frac{p_1(Z)}{p^*(Z)} \right] 
= \int_{-\infty}^\infty \frac{2 p_1(z) q(z)}{p_1(z) + p_2(z)} dz
= \E_\Q\left[ \frac{p_2(Z)}{p^*(Z)} \right].
\]
Since $p_1 + p_2 = 2p^*$ we also have $\E_\Q[p_1(Z)/p^*(Z)] + \E_\Q[p_2(Z)/p^*(Z)] = 2$, and conclude that
\[
\E_\Q\left[ \frac{p_i(Z)}{p^*(Z)} \right] = 1, \quad i=1,2.
\]
Thus \eqref{eq_dual_numeraire_propoerty_mu} of Theorem~\ref{T_ripr_mu} is satisfied, and $\P^*$ is the RIPr. The numeraire is
\begin{equation} \label{eq_ex_param_num}
X^* = \frac{2 q(Z)}{p_1(Z) + p_2(Z)}.
\end{equation}
It is worth remarking that the universal inference e-variable~\citep{wasserman2020universal} would be given by $q(Z)/\max\{p_1(Z),p_2(Z)\}$, which is clearly smaller than $X^*$. Note also that we could have taken $\Q$ to be any symmetric distribution with a strictly positive density. More generally, we could let $\Q$ be symmetric with a density $q$ that is not necessarily strictly positive. In this case the RIPr would be a sub-probability measure with density $p^* = \frac{1}{2}(p_1+p_2)\bm1_{\{q > 0\}}$, and the numeraire would still be given by \eqref{eq_ex_param_num}.

\subsection{Bounded mean}

This example is a variant of one studied in \cite{waudby2020estimating}, with the corresponding duality theory derived explicitly in~\cite{honda2010asymptotically}, and generalized to other classes of distributions in~\cite{agrawal2020optimal}. 
Let $\Zcal = [0,1]$ so that the random variable $Z$ takes values in the unit interval. Fix $\mu \in (0,\frac12)$ and consider the null hypothesis that the mean of $Z$ is at most $\mu$,
\[
\Pcal = \{\P \in M_1 \colon \E_\P[Z] \le \mu\}.
\]
There is no single dominating measure for $\Pcal$ since it contains the uncountable non-dominated family $\{\delta_z \colon z \in [0,\mu]\}$. However, $\Pcal$ is generated by $\Ecal_0 = \{Z/\mu\}$ in the sense of Corollary~\ref{C:240207}, so our strategy will be to locate a candidate numeraire and then apply the corollary to verify that the candidate is, in fact, the numeraire. The alternative hypothesis $\Q$ is the uniform distribution on $[0,1]$, and we have $\Q \ll \Pcal$ for the simple reason that $\P(A) = 0$ for all $\P \in \Pcal$ implies that $A$ must actually be empty.

To find a candidate numeraire, we observe that two natural e-variables are $Z/\mu$ and the constant one. All convex combinations of these are also e-variables; equivalently, $1 + \lambda(Z-\mu)$ is an e-variable for each $\lambda \in [0,\mu^{-1}]$. We now look for a log-optimal e-variable in this class by directly maximizing $f(\lambda) = \E_\Q[\log (1 + \lambda(Z-\mu))]$ over $\lambda \in [0,\mu^{-1}]$. This is a strictly concave function whose derivative $f'(\lambda) = \E_\Q[(Z-\mu) / (1 + \lambda(Z-\mu))]$ satisfies $f'(0) = \E_\Q[Z] - \mu > 0$ and 
$f'(\mu^{-1}) = \E_\Q[\mu - \mu^2/Z] = -\infty$. Thus there is a unique interior maximizer $\lambda^* \in (0,\mu^{-1})$, which is characterized by the first-order condition
\begin{equation} \label{eq_bdd_mean_testing_foc}
\E_\Q\left[ \frac{Z - \mu}{1 + \lambda^* (Z - \mu)} \right] = 0.
\end{equation}
Since $\Q$ is the standard uniform distribution we can be more explicit. Nothing changes if we first multiply both sides by $\lambda^*$, and then the left-hand side becomes
\[
\int_0^1 \frac{\lambda^*(z - \mu)}{1 + \lambda^* (z - \mu)} dz
= 1 -  \int_0^1 \frac{1}{1 + \lambda^* (z - \mu)} dz
= 1 - \frac{1}{\lambda^*}\log\left(\frac{1 + \lambda^* (1 - \mu)}{1 - \lambda^*\mu} \right).
\]
Thus \eqref{eq_bdd_mean_testing_foc} for $\lambda^* \in (0,\mu^{-1})$ is equivalent to
\[
    \frac{1 + \lambda^* (1 - \mu)}{1 - \lambda^*\mu} = e^{\lambda^*},
\]
which is easily solved numerically.
This leads us to the candidate numeraire
\[
X^* = 1 + \lambda^* (Z - \mu),
\]
which is strictly positive and finite. To  verify that this is indeed the numeraire, we use \eqref{eq_bdd_mean_testing_foc}
to get, for any e-variable of the form $X = 1 + \lambda (Z - \mu) = X^* + (\lambda - \lambda^*)(Z-\mu)$, that 
\begin{equation} \label{eq_ex_bdd_mean_num_check}
\E_\Q\left[ \frac{X}{X^*} \right] = 1 + (\lambda - \lambda^*) \E_\Q\left[ \frac{Z - \mu}{1 + \lambda^* (Z - \mu)} \right] = 1.
\end{equation}
Taking $\lambda = 0$ and $\lambda = 1/\mu$ we see that \eqref{T_verification_2} of Corollary~\ref{C:240207} is satisfied and, hence, that $X^*$ is the numeraire and the RIPr $\P^*$  belongs to $\Pcal$.

\begin{remark}
For the null hypothesis considered here, we conjecture that the set of all e-variables is given by
\[
\Ecal = \{X \in F_+ \colon X \le 1 + \lambda (Z - \mu) \text{ for some } \lambda \in [0,\mu^{-1}] \}.
\]
It is clear that every such $X$ is an e-variable, but it is less obvious that every e-variable is of this form. Nonetheless, equipped with this knowledge the calculation in \eqref{eq_ex_bdd_mean_num_check} would directly yield the numeraire property of $X^*$. The usefulness of Corollary~\ref{C:240207} is that a complete description of $\Ecal$ is not required; it is enough to check \eqref{eq_ex_bdd_mean_num_check} for $X=1$ and for $X$ in the generating set $\Ecal_0$.
\end{remark}

\subsection{Sub-Gaussian with nonpositive mean}

Now take $\Zcal = \R$ and consider the null hypothesis that the observation $Z$ has a 1-sub-Gaussian distribution with nonpositive mean. That is, we set
\[
\Pcal = \left\{\P \in M_1 \colon \E_\P[ e^{\lambda Z - \lambda^2/2} ] \le 1 \text{ for all } \lambda \in [0,\infty) \right\}.
\]
The `one-sided' restriction $\lambda \in [0,\infty)$ implies that $Z$ has nonpositive (potentially nonzero and even infinite) mean under any $\P \in \Pcal$. Indeed, monotone convergence and the definition of $\Pcal$ yield $\E_\P[Z] = \lim_{\lambda \downarrow 0} \E_\P[ (e^{\lambda Z} - 1)/\lambda ] \le \lim_{\lambda \downarrow 0} (e^{\lambda^2/2} - 1)/\lambda = 0$. As in the previous example, $\Pcal$ does not admit any dominating measure. It is generated by the family $\Ecal_0 = \{e^{\lambda Z - \lambda^2/2} \colon \lambda \in [0,\infty) \}$, so we will again look for a candidate numeraire and verify it using Corollary~\ref{C:240207}. We let the alternative hypothesis $\Q$ be normal with mean $\mu > 0$ and unit variance. As in the previous example, and for the same reason, we have $\Q \ll \Pcal$.

To find a candidate numeraire we maximize $\E_\Q[\log X]$ over $X \in \Ecal_0$. That is, we maximize $\E_\Q[ \lambda Z - \lambda^2/2] = \lambda \mu - \lambda^2/2$ over $\lambda \in [0,\infty)$. The maximizer is $\lambda^* = \mu$, which yields the candidate
\[
X^* = e^{\mu Z - \mu^2/2}.
\]
This is finite and strictly positive. Moreover, \eqref{T_verification_2} of Corollary~\ref{C:240207} is satisfied because
\begin{equation} \label{eq_sub_gaussian_testing_foc}
\E_\Q\left[ \frac{e^{\lambda Z - \lambda^2/2}}{e^{\mu Z - \mu^2/2}} \right] = \E_\Q\left[ e^{(\lambda - \mu) (Z - \mu) - (\lambda - \mu)^2/2} \right] = 1.
\end{equation}
We conclude that $X^*$ is the numeraire and, consequently, that the RIPr $\P^*$ is the standard normal distribution. This example can be easily generalized to $\sigma$-sub-Gaussian distributions for $\sigma \neq 1$, but we omit this for brevity.

\begin{remark}
We conjecture that the set of all e-variables is given by
\[
\Ecal = \left\{X \in F_+ \colon X \le \int_{[0,\infty)} e^{\lambda Z - \lambda^2/2} \pi(d\lambda) \text{ for some probability measure $\pi$ on $[0,\infty)$} \right\}.
\]
Although it is clear that every such $X$ is an e-variable, showing that this really includes all e-variables seems to require sophisticated functional analytic methods that are beyond the scope of this paper. Fortunately, Corollary~\ref{C:240207} sidesteps this issue. Nonetheless, taking the description of $\Ecal$ for granted the numeraire property can be checked directly. Indeed, for a general e-variable $X \le \int_{[0,\infty)} e^{\lambda Z - \lambda^2/2} \pi(d\lambda)$, where $\pi$ is a probability measure on $[0,\infty)$, Tonelli's theorem and \eqref{eq_sub_gaussian_testing_foc} give
\[
\E_\Q\left[ \frac{X}{X^*} \right] \le \int_{[0,\infty)} \E_\Q\left[ \frac{e^{\lambda Z - \lambda^2/2}}{e^{\mu Z - \mu^2/2}} \right] \pi(d\lambda) = 1.
\]
Finally, let us mention that \cite{ramdas2020admissible} prove that $e^{\lambda Z - \lambda^2/2}$ is not just an e-variable, but an admissible one for every $\lambda \geq 0$.
\end{remark}

\subsection{Symmetric distributions} \label{S_symmetric_laws}

Continue to take $\Zcal = \R$. We now consider the null hypothesis that $Z$ is symmetric,
\[
\Pcal = \left\{ \P \in M_1 \colon Z \text{ and } -Z \text{ have the same distribution under } \P \right\},
\]
which is again a non-dominated family. We also fix an alternative hypothesis $\Q$ that admits a Lebesgue density $q$. It is natural to conjecture that the RIPr is given by the symmetrization $\widetilde \P$ of $\Q$, whose density is $\widetilde p(z) = (q(z) + q(-z))/2$. However, this cannot quite be true in general because $\widetilde\P$ need not be equivalent to $\Q$. Instead, we claim that the RIPr is the measure $\P^*$ with density
\[
p^*(z) = \frac{1}{2}\left( q(z) + q(-z) \right) \bm1_{\{q(z) > 0\}}.
\]
This is the absolutely continuous part of $\widetilde\P$ with respect to $\Q$. It is a probability measure if $\Q$ has symmetric support, and otherwise a proper sub-probability measure. Note that $\widetilde\P$ belongs to $\Pcal$, which in particular shows that $\Q \ll \Pcal$ since $\Q \ll \widetilde\P$. (Alternatively, we again have that $\P(A) = 0$ for all $\P \in \Pcal$ implies that $A$ is empty, which also yields $\Q \ll \Pcal$.) 

To check that $\P^*$ is the RIPr we will show that the implied candidate numeraire is, in fact, the numeraire. It is given by
\[
X^* = \frac{d\Q}{d\P^*} = \frac{2q(Z)}{q(Z) + q(-Z)}.
\]
This expression also appears in~\citet[Section 5--6]{koning2023post}.
We claim that the set of all e-variables is
\[
\Ecal = \left\{X \in F_+ \colon X \le 1 + \phi(Z) \text{ for some odd function } \phi \right\},
\]
where we recall that a function $\phi$ is odd if $\phi(z) + \phi(-z) = 0$ for all $z \in \R$. Indeed, any $X$ of this form satisfies $\E_\P[X] \le 1 + \E_\P[\phi(Z)] = 1$, where the symmetry of $Z$ and the oddness of $\phi$ were used to get $\E_\P[\phi(Z)] = \E_\P[\phi(-Z)] = -\E_\P[\phi(Z)]$ and hence $\E_\P[\phi(Z)] = 0$. Conversely, for any e-variable $X = f(Z)$ we may write $X = \frac12 (f(Z) + f(-Z)) + \phi(Z)$ where $\phi(z) = \frac12(f(z) - f(-z))$ is the odd part of $f(z)$. For any $z \in \R$ we use the symmetric distribution $\P = \frac12 (\delta_z + \delta_{-z})$ and the fact that $X$ is an e-variable to get $\frac12 (f(z) + f(-z)) = \E_\P[f(Z)] \le 1$. Hence $X \le 1 + \phi(Z)$ as required. This is a different conclusion from the fact that every \emph{admissible and exact} e-variable is of the form $1+\phi(Z)$ for some odd function $\phi$, as shown by \cite{ramdas2020admissible}.

We can now verify the numeraire property. First, $X^*$ is $\Q$-almost surely strictly positive and finite, and it is an e-variable because it can be written as $X^* = 1 + \phi^*(Z)$ where
\[
\phi^*(z) = \frac{q(z) - q(-z)}{q(z) + q(-z)}
\]
is odd. Next, for any e-variable $X \le 1 + \phi(Z)$ where $\phi$ is odd we get
\[
\E_\Q\left[ \frac{X}{X^*} \right] \le \E_\Q\left[ (1 + \phi(Z)) \frac{q(Z) + q(-Z)}{2 q(Z)} \right]
= \E_{\widetilde\P}\left[(1 + \phi(Z)) \mathbf{1}_{\{q(Z) > 0\}} \right]
\le   \E_{\widetilde\P}[1 + \phi(Z)] = 1.
\]
This confirms the numeraire property. Finally, let us remark that it is not really necessary that $\Q$ admit a density. The symmetrization of $\Q$ is still well-defined as $\widetilde\P = \frac12 (\Q + \widetilde\Q)$ where $\widetilde \Q$ is the distribution of $-Z$ under $\Q$, or equivalently, the pushforward of $\Q$ under the reflection map $z \mapsto -z$. The RIPr is then the absolutely continuous part $\P^* = \widetilde \P^a$, and the numeraire is any nonnegative version of the Radon--Nikodym derivative $d\Q/d\P^*$ such that $d\Q/d\P^* - 1$ is odd.

\subsection{Exponential family with one-dimensional sufficient statistic}

Here $\Zcal$ can be general and is equipped with a reference measure $\mu$. The following example shows how to use the results in Section~\ref{S_ref_measure} to obtain the result in \cite[Example~4]{grunwald2020safe}. Consider an exponential family of densities
\[
p_\theta(z) = e^{\theta T(z) -  A(\theta)} 
\]
with respect to $\mu$, where $A$ is convex and differentiable, the sufficient statistic $T(z)$ is one-dimensional, and the natural parameter $\theta$ ranges in some interval $\Theta \subset \R$. The null hypothesis is $\Pcal = \{p_\theta \colon \theta \in \Theta_0\}$ for some closed subset $\Theta_0 \subset \Theta$, and the alternative is $q = p_{\theta_1}$ for some $\theta_1 \in \Theta$. We suppose that $\Theta_0$ has a smallest element $\theta^*$ and that $\theta_1 < \theta^*$. It is natural to conjecture that $p_{\theta^*}$ is the RIPr. To confirm this, it suffices to verify \eqref{eq_dual_numeraire_propoerty_mu}. Using the standard formula for the moment generating function of the sufficient statistic we get
\[
\E_\Q\left[\frac{p_\theta}{p_{\theta^*}}\right] = \E_\Q\left[e^{(\theta-\theta^*) T(Z) - A(\theta) + A(\theta^*)}\right] = e^{A(\theta_1+\theta-\theta^*) - A(\theta_1) - A(\theta) + A(\theta^*)}
\]
for all $\theta \in \Theta_0$. Since $A$ is convex, its derivative $A'$ is increasing. Together with the fundamental theorem of calculus, as well as $\theta_1 < \theta^*$ and $\theta-\theta^* \ge 0$ for $\theta \in \Theta_0$, this gives
\begin{align*}
A(\theta_1+\theta-\theta^*) - A(\theta_1)
&= \int_0^1 (\theta - \theta^*) A'(\theta_1 + t(\theta - \theta^*)) dt \\
&\le \int_0^1 (\theta - \theta^*) A'(\theta^* + t(\theta - \theta^*)) dt \\
&= A(\theta) - A(\theta^*).
\end{align*}
We conclude that $\E_\Q[p_\theta / p_{\theta^*}] \le 1$ for all $\theta \in \Theta_0$ so that \eqref{eq_dual_numeraire_propoerty_mu} holds and $p_{\theta^*}$ is indeed the RIPr. As a result the numeraire is the likelihood ratio
\[
X^* = \frac{p_{\theta_1}(Z)}{p_{\theta^*}(Z)} = e^{(\theta_1-\theta^*) T(Z) - A(\theta_1) + A(\theta^*)}.
\]
We refer to \cite{1201070} for an extensive study of information projections involving exponential families.

\begin{remark}
When $\Pcal$ is a parametric family of distributions it is often easier to first find the RIPr and then use it to find the numeraire. Here we used Theorem~\ref{T_ripr_mu} to confirm that the candidate $p_{\theta^*}$ was indeed the RIPr. We could also have minimized relative entropy and used Corollary~\ref{C_min_rel_entropy}. Since the corollary assumes that $\Pcal$ is convex, which fails here, we would have had to apply it with the convex hull of $\Pcal$ and perform the minimization over this set.
\end{remark}

\section{Beyond the logarithm: reverse R\'enyi projection} \label{S_renyi}

The duality between log-utility and relative entropy can be extended to other utilities and divergences. One case where this can be done rigorously using the results presented so far is the power utility
\[
U(x) = \frac{x^{1-\gamma}}{1-\gamma}
\]
for $\gamma > 1$. In this section we outline the resulting theory using a well-known method from mathematical finance pioneered by \cite{MR1089152} and \cite{MR1722287}. However, in contrast to that literature we work as in Sections~\ref{S_numeraire} and~\ref{S_duality} with completely arbitrary $\Pcal$ and $\Q$.

The utility function $U$ is continuous, increasing, concave, differentiable, and bounded from above by zero. Lemma~\ref{L_u_max} is therefore applicable and yields an optimal e-variable for the maximization problem
\[
\sup_{X \in \Ecal} \E_\Q[U(X)].
\]
We denote the maximizer by $X^*_\gamma$ and call it the \emph{$U$-optimal e-variable}. Observe that
\[
\E_\Q[U'(X^*_\gamma) X^*_\gamma] = \E_\Q\left[(X^*_\gamma)^{1-\gamma}\right] = (1-\gamma) \E_\Q[U(X^*_\gamma)] \le (1-\gamma) U(1) = 1,
\]
where we use the convention $U'(\infty) \infty = 0$ and where the inequality relies on the fact that the constant e-variable $X=1$ is suboptimal and that $\gamma > 1$. Thus Lemma~\ref{L_u_max} further implies the first-order condition
\begin{equation} \label{eq_foc_power_utility}
\E_\Q\left[ (X^*_\gamma)^{-\gamma} X \right] \le \E_\Q\left[ (X^*_\gamma)^{1-\gamma} \right], \quad X \in \Ecal,
\end{equation}
where $(X^*_\gamma)^{-\gamma} X$ is understood as zero on $\{X^*_\gamma = \infty\}$. This will allow us to identify an analog of the RIPr which we denote by $\P^*_\gamma$. In the fully degenerate case where $X^*_\gamma = \infty$, $\Q$-almost surely, we simply set $\P^*_\gamma = 0$. In the more interesting case where $X^*_\gamma$ is finite with positive $\Q$-probability, we let $\P^*_\gamma$ be the measure whose density with respect to $\Q$ is
\begin{equation} \label{eq_P_star_renyi}
\frac{d\P^*_\gamma}{d\Q} = \frac{(X^*_\gamma)^{-\gamma}}{\E_\Q\left[ (X^*_\gamma)^{1-\gamma} \right]}.
\end{equation}
Thanks to \eqref{eq_foc_power_utility}, $\P^*_\gamma$ belongs to $\Peff$. In analogy with \eqref{eq_strong_duality} we expect $\P^*_\gamma$ to minimize a suitable dual objective function. To discover its form, we look for a weak duality inequality and introduce the Legendre transform $V(y) = \sup_{x>0} \{U(x) - xy\}$. Evaluating the supremum shows that
\[
V(y) = - \frac{y^{1-1/\gamma}}{1-1/\gamma}.
\]
Now, for any e-variable $X$, any $\P \in \Peff$, and any $y \ge 0$, we have
\[
\begin{aligned}
\E_\Q[U(X)] &\le \E_\Q[U(X)] + y(1 - \E_\P[X]) \\
&\le \E_\Q\left[U(X) - y \frac{d\P^a}{d\Q} X \right] + y\\
&\le \E_\Q\left[V\left(y \frac{d\P^a}{d\Q}\right)\right] + y \\
&= - \frac{y^{1-1/\gamma}}{1-1/\gamma} \E_\Q\left[\left(\frac{d\P^a}{d\Q}\right)^{1-1/\gamma}\right] + y.
\end{aligned}
\]
To obtain a tight bound we minimize the right-hand side over $y$. A computation shows that the minimizer is $y = \E_\Q[(d\P^a/d\Q)^{1-1/\gamma}]^\gamma$, which is finite, possibly zero, thanks to Jensen's inequality. With this choice we get, for any e-variable $X$ and any $\P \in \Peff$, the weak duality inequality
\begin{equation*} 
\E_\Q[U(X)] \le \frac{1}{1-\gamma} \E_\Q\left[\left(\frac{d\P^a}{d\Q}\right)^{1-1/\gamma}\right]^\gamma.
\end{equation*}
Equality is achieved by taking $X = X^*_\gamma$ and $\P = \P^*_\gamma$, so we conclude that $\P^*_\gamma$ solves the minimization problem
\[
\inf_{\P \in \Peff} \frac{1}{1-\gamma} \E_\Q\left[\left(\frac{d\P^a}{d\Q}\right)^{1-1/\gamma}\right]^\gamma
\]
with optimal value equal to $\E_\Q[U(X^*_\gamma)]$. Observe that the quantity being minimized is an increasing function of the \emph{R\'enyi divergence} of order $1/\gamma$ of $\Q$ from $\P$,
\[
R_{1/\gamma}(\Q \mid \P) = \frac{1}{1/\gamma - 1} \log \E_\Q\left[\left(\frac{d\P^a}{d\Q}\right)^{1-1/\gamma}\right].
\]
Thus $\P^*_\gamma$ also minimizes the R\'enyi divergence, and as such it is known as the \emph{reverse R\'enyi projection} of order $1/\gamma$ of $\Q$ on $\Peff$. In particular, we see that the reverse R\'enyi projection always exists if the order $1/\gamma$ is in $(0,1)$. We summarize the above observations, along with a few others, in the following theorem.

\begin{theorem} \label{T_reverse_renyi}
Let $\gamma > 1$ and $U(x) = x^{1-\gamma}/(1-\gamma)$. Then
\begin{equation} \label{eq_T_reverse_renyi}
\sup_{X \in \Ecal} \E_\Q[U(X)]
= \frac{1}{1-\gamma} \exp\left( (1 - \gamma) \inf_{\P \in \Peff} R_{1/\gamma}(\Q \mid \P) \right),
\end{equation}
and the supremum and infimum are attained by $X^*_\gamma \in \Ecal$ and $\P^*_\gamma \in \Peff$ related by \eqref{eq_P_star_renyi}. Furthermore, $X^*_\gamma$ is the $\Q$-almost surely unique maximizer, and $\P^*_\gamma$ is the unique minimizer among the elements of $\Peff$ that are absolutely continuous with respect to $\Q$. Finally, $X^*_\gamma$ can be expressed in terms of $\P^*_\gamma$ by
\begin{equation} \label{eq_T_reverse_renyi_2}
X^*_\gamma = \left(\frac{d\P^*_\gamma}{d\Q}\right)^{-1/\gamma} \biggm/ \E_\Q\left[\left(\frac{d\P^*_\gamma}{d\Q}\right)^{1-1/\gamma}\right],
\end{equation}
which is understood as $+\infty$ on the set where $d\P^*_\gamma / d\Q = 0$.
\end{theorem}

\begin{proof}
The equality in \eqref{eq_T_reverse_renyi} and the form of the solutions were shown above. The $\Q$-almost sure uniqueness of $X^*_\gamma$ follows from strict concavity of the function $f(X) = \E_\Q[U(X)]$ on $\Ecal$. Since $\gamma > 1$, the R\'enyi divergence is a decreasing transformation of the function $g(\P) = \E_\Q[(d\P^a/d\Q)^{1-1/\gamma}]$, which is thus maximized by $\P^*_\gamma$. The function $g(\P)$ is strictly concave on the set of elements of $\Peff$ that are absolutely continuous with respect to $\Q$. Thus $\P^*_\gamma$ is the unique maximizer of $g(\P)$, and the unique minimizer of the R\'enyi divergence. Finally, \eqref{eq_T_reverse_renyi_2} is obtained by noting that
\[
\E_\Q\left[ (X^*_\gamma)^{1-\gamma} \right]
= \E_\Q\left[\left(\frac{d\P^*_\gamma}{d\Q}\right)^{1-1/\gamma}\right]^\gamma
\]
and then inverting \eqref{eq_P_star_renyi}.
\end{proof}

\begin{remark}
\citet[Section~V]{harremoes2023universal} mention the e-variable $(d\P^*_\gamma/d\Q)^{-1/\gamma}$ constructed using the reverse R\'enyi projection. This e-variable is strictly smaller than $X^*_\gamma$, because the denominator in \eqref{eq_T_reverse_renyi_2} is strictly smaller than one by Jensen's inequality.
\end{remark}

\begin{remark}
The method described above can be applied to more general utility functions, resulting in projections associated with more general divergences. An interesting question is whether all $f$-divergences can be covered, at least under suitable technical assumptions.
\end{remark}

\section{Universal inference is inadmissible} \label{S_comp_alt_univ_inf}

We assume that there exists a common reference measure $\mu$, and follow the notational conventions introduced before Theorem~\ref{T_ripr_mu} to identify distributions with their respective densities, written in lowercase. For a point alternative $q$ over the data $Z$, the method of universal inference~\citep{wasserman2020universal} boils down to constructing the e-variable $X^\text{UI} = q(Z)/p_\text{max}(Z)$ where $p_\text{max}(Z) = \sup_{p \in \Pcal} p(Z)$ is the maximum likelihood. (Strictly speaking, the supremum here should be understood as an essential supremum under the reference measure.) To check that this is indeed an e-variable, simply note that for any $\P \in \Pcal$ with density $p$ we have
\[
\E_\P \left[\frac{q(Z)}{p_\text{max}(Z)}\right] \leq \E_\P \left[\frac{q(Z)}{p(Z)}\right] = 1.
\]
To compare the universal inference e-variable with the numeraire $X^* = q(Z)/p^*(Z)$, we first claim that the RIPr satisfies $p^*(Z) \le p_\text{max}(Z)$ up to $\mu$-nullsets. This is obvious if the RIPr belongs to $\Pcal$, but in general it only belongs to $\Peff$. In this case the claim follows by applying the following lemma with $f = p_\text{max}$.

\begin{lemma}
Let $f$ be a random variable such that $f \ge p$, $\mu$-almost surely, for all $p \in \Pcal$. Then $f \ge p$, $\mu$-almost surely, for all $p \in \Peff$.
\end{lemma}

\begin{proof}
It is clear that $f \ge p$, $\mu$-almost surely, for all sub-probability densities $p$ dominated by some element of the convex hull of $\Pcal$, and then also for all $p$ in the $\mu$-probability closure of this set. This was called $\Pcal''$ in the proof of Theorem~\ref{T_ripr_mu}, where it was shown that $\Pcal''$ is closed, convex, and solid. By Lemma~\ref{L_P_bipolar_char_ref_measure}, $\Pcal''$ contains $\Peff$, completing the proof.
\end{proof}

We thus see that $p^*(Z) \le p_\text{max}(Z)$, and hence $X^* \ge X^\text{UI}$, up to $\mu$-nullsets. In nondegenerate situations involving a composite null hypothesis, the inequality will be strict with positive $\mu$-probability. Thus, in such cases, the relatively general method of universal inference is in fact inadmissible. We end by noting that our numeraire imposes weaker assumptions than the universal inference work, which needs a reference measure to define likelihoods.

\section{Summary} \label{S_summary}

We established that under no assumptions on the composite null $\Pcal$ and point alternative $\Q$, there exists a unique special e-variable $X^*$, the numeraire, which satisfies $\E_\Q[X/X^*]\leq 1$ for every other e-variable $X$, and is also log-optimal. Further, $X^*$ can be written as a likelihood ratio $d\Q/d\P^*$ for a unique (sub-)probability distribution $\P^*$, absolutely continuous with respect to $\Q$ and belonging to the effective null hypothesis $\Peff$ associated with $\Pcal$. The numeraire is the only e-variable of this form. A strong duality theory establishes that  $\E_\Q[\log X^*]$ equals the minimum relative entropy of $\Q$ to $\Peff$, achieved by $\P^*$. This fully generalizes (to composite nulls) the fact that the likelihood ratio is the log-optimal e-variable for point nulls. We gave several sufficient conditions to identify and certify the numeraire, and showed that these were easy to verify in several nonparametric examples without a reference measure that were out of the reach of earlier methods. We also showed how to generalize our theory beyond log-optimality. 

One natural direction for future work is to extend our theory to composite alternatives, using the \emph{method of mixtures} or the \emph{plug-in method}. A second direction is to extend our theory to the sequential setting using nonnegative supermartingales and \emph{e-processes} (sequential generalizations of e-variables). See~\cite{ramdas2023game} for an elaboration of the italicized terms and existing work. The second topic introduces additional subtleties through dealing with filtrations and stopping times, a direction we are currently pursuing.

\subsection*{Acknowledgments} We thank the  referees for insightful comments and helpful suggestions.

\bibliography{bibliography}

\begin{thebibliography}{34}
\providecommand{\natexlab}[1]{#1}
\providecommand{\url}[1]{\texttt{#1}}
\expandafter\ifx\csname urlstyle\endcsname\relax
  \providecommand{\doi}[1]{doi: #1}\else
  \providecommand{\doi}{doi: \begingroup \urlstyle{rm}\Url}\fi

\bibitem[Agrawal et~al.(2020)Agrawal, Juneja, and Glynn]{agrawal2020optimal}
Shubhada Agrawal, Sandeep Juneja, and Peter Glynn.
\newblock Optimal $\delta$-correct best-arm selection for heavy-tailed
  distributions.
\newblock In \emph{Algorithmic Learning Theory}, pages 61--110. PMLR, 2020.

\bibitem[Bartl and Kupper(2019)]{MR3910414}
Daniel Bartl and Michael Kupper.
\newblock A pointwise bipolar theorem.
\newblock \emph{Proc. Amer. Math. Soc.}, 147\penalty0 (4):\penalty0 1483--1495,
  2019.
\newblock ISSN 0002-9939,1088-6826.
\newblock \doi{10.1090/proc/14231}.
\newblock URL \url{https://doi.org/10.1090/proc/14231}.

\bibitem[Becherer(2001)]{MR1849424}
Dirk Becherer.
\newblock The numeraire portfolio for unbounded semimartingales.
\newblock \emph{Finance Stoch.}, 5\penalty0 (3):\penalty0 327--341, 2001.
\newblock ISSN 0949-2984,1432-1122.
\newblock \doi{10.1007/PL00013535}.
\newblock URL \url{https://doi.org/10.1007/PL00013535}.

\bibitem[Brannath and Schachermayer(1999)]{Brannath}
W.~Brannath and W.~Schachermayer.
\newblock A bipolar theorem for {$L^0_+(\Omega, \mathcal{F},\mathbb{P})$}.
\newblock In \emph{S\'{e}minaire de {P}robabilit\'{e}s, {XXXIII}}, volume 1709
  of \emph{Lecture Notes in Math.}, pages 349--354. Springer, Berlin, 1999.
\newblock ISBN 3-540-66342-8.
\newblock \doi{10.1007/BFb0096525}.
\newblock URL \url{https://doi.org/10.1007/BFb0096525}.

\bibitem[Casgrain et~al.(2024)Casgrain, Larsson, and Ziegel]{MR4699555}
Philippe Casgrain, Martin Larsson, and Johanna Ziegel.
\newblock Sequential testing for elicitable functionals via supermartingales.
\newblock \emph{Bernoulli}, 30\penalty0 (2):\penalty0 1347--1374, 2024.
\newblock ISSN 1350-7265,1573-9759.
\newblock \doi{10.3150/23-bej1634}.
\newblock URL \url{https://doi.org/10.3150/23-bej1634}.

\bibitem[Cover and Thomas(2006)]{MR2239987}
Thomas~M. Cover and Joy~A. Thomas.
\newblock \emph{Elements of Information Theory}.
\newblock Wiley-Interscience [John Wiley \& Sons], Hoboken, NJ, second edition,
  2006.
\newblock ISBN 978-0-471-24195-9; 0-471-24195-4.

\bibitem[Csiszar and Matus(2003)]{1201070}
I.~Csiszar and F.~Matus.
\newblock Information projections revisited.
\newblock \emph{IEEE Transactions on Information Theory}, 49\penalty0
  (6):\penalty0 1474--1490, 2003.
\newblock \doi{10.1109/TIT.2003.810633}.

\bibitem[Csisz\'{a}r and Tusn\'{a}dy(1984)]{csiszar1984information}
I.~Csisz\'{a}r and G.~Tusn\'{a}dy.
\newblock Information geometry and alternating minimization procedures.
\newblock \emph{Statist. Decisions}, pages 205--237, 1984.
\newblock ISSN 0721-2631.
\newblock Recent Results in Estimation Theory and Related Topics.

\bibitem[Cvitani\'{c} and Karatzas(2001)]{Cvitanic:Karatzas:2001}
Jak\v{s}a Cvitani\'{c} and Ioannis Karatzas.
\newblock Generalized {N}eyman-{P}earson lemma via convex duality.
\newblock \emph{Bernoulli}, 7\penalty0 (1):\penalty0 79--97, 2001.
\newblock ISSN 1350-7265,1573-9759.
\newblock \doi{10.2307/3318603}.
\newblock URL \url{https://doi.org/10.2307/3318603}.

\bibitem[Delbaen and Schachermayer(1994)]{DS:1994}
Freddy Delbaen and Walter Schachermayer.
\newblock A general version of the fundamental theorem of asset pricing.
\newblock \emph{Math. Ann.}, 300\penalty0 (3):\penalty0 463--520, 1994.
\newblock ISSN 0025-5831,1432-1807.
\newblock \doi{10.1007/BF01450498}.
\newblock URL \url{https://doi.org/10.1007/BF01450498}.

\bibitem[Ekeland and T\'emam(1999)]{MR1727362}
Ivar Ekeland and Roger T\'emam.
\newblock \emph{Convex Analysis and Variational Problems}, volume~28 of
  \emph{Classics in Applied Mathematics}.
\newblock Society for Industrial and Applied Mathematics (SIAM), Philadelphia,
  PA, english edition, 1999.
\newblock ISBN 0-89871-450-8.
\newblock \doi{10.1137/1.9781611971088}.
\newblock URL \url{https://doi.org/10.1137/1.9781611971088}.
\newblock Translated from the French.

\bibitem[Gr{\"u}nwald et~al.(2024)Gr{\"u}nwald, de~Heide, and
  Koolen]{grunwald2020safe}
Peter Gr{\"u}nwald, Rianne de~Heide, and Wouter~M Koolen.
\newblock Safe testing.
\newblock \emph{Journal of the Royal Statistical Society, Series B
  (Methodology), with discussion}, 2024.

\bibitem[Halmos and Savage(1949)]{halmos1949application}
Paul~R Halmos and Leonard~J Savage.
\newblock Application of the {Radon-Nikodym} theorem to the theory of
  sufficient statistics.
\newblock \emph{The Annals of Mathematical Statistics}, 20\penalty0
  (2):\penalty0 225--241, 1949.

\bibitem[Honda and Takemura(2010)]{honda2010asymptotically}
Junya Honda and Akimichi Takemura.
\newblock An asymptotically optimal bandit algorithm for bounded support
  models.
\newblock In \emph{Conference on Learning Theory}, pages 67--79. Citeseer,
  2010.

\bibitem[Howard et~al.(2020)Howard, Ramdas, McAuliffe, and
  Sekhon]{howard2020time}
Steven~R Howard, Aaditya Ramdas, Jon McAuliffe, and Jasjeet Sekhon.
\newblock Time-uniform {C}hernoff bounds via nonnegative supermartingales.
\newblock \emph{Probability Surveys}, 2020.

\bibitem[Howard et~al.(2021)Howard, Ramdas, McAuliffe, and
  Sekhon]{howard2021time}
Steven~R Howard, Aaditya Ramdas, Jon McAuliffe, and Jasjeet Sekhon.
\newblock Time-uniform, nonparametric, nonasymptotic confidence sequences.
\newblock \emph{The Annals of Statistics}, 2021.

\bibitem[Huber and Strassen(1973)]{Huber:Strassen:1973}
Peter~J. Huber and Volker Strassen.
\newblock Minimax tests and the {N}eyman-{P}earson lemma for capacities.
\newblock \emph{Ann. Statist.}, 1:\penalty0 251--263, 1973.
\newblock ISSN 0090-5364,2168-8966.
\newblock URL
  \url{http://links.jstor.org/sici?sici=0090-5364(197303)1:2<251:MTATNL>2.0.CO;2-R&origin=MSN}.

\bibitem[Karatzas and Kardaras(2007)]{MR2335830}
Ioannis Karatzas and Constantinos Kardaras.
\newblock The num\'{e}raire portfolio in semimartingale financial models.
\newblock \emph{Finance Stoch.}, 11\penalty0 (4):\penalty0 447--493, 2007.
\newblock ISSN 0949-2984,1432-1122.
\newblock \doi{10.1007/s00780-007-0047-3}.
\newblock URL \url{https://doi.org/10.1007/s00780-007-0047-3}.

\bibitem[Karatzas and Kardaras(2021)]{KK_book}
Ioannis Karatzas and Constantinos Kardaras.
\newblock \emph{Portfolio Theory and Arbitrage---A Course in Mathematical
  Finance}, volume 214 of \emph{Graduate Studies in Mathematics}.
\newblock American Mathematical Society, Providence, RI, 2021.
\newblock ISBN 978-1-4704-6014-3.
\newblock \doi{10.1090/gsm/214}.
\newblock URL \url{https://doi.org/10.1090/gsm/214}.

\bibitem[Karatzas et~al.(1991)Karatzas, Lehoczky, Shreve, and Xu]{MR1089152}
Ioannis Karatzas, John~P. Lehoczky, Steven~E. Shreve, and Gan-Lin Xu.
\newblock Martingale and duality methods for utility maximization in an
  incomplete market.
\newblock \emph{SIAM J. Control Optim.}, 29\penalty0 (3):\penalty0 702--730,
  1991.
\newblock ISSN 0363-0129.
\newblock \doi{10.1137/0329039}.
\newblock URL \url{https://doi.org/10.1137/0329039}.

\bibitem[Koning(2023)]{koning2023post}
Nick~W Koning.
\newblock Post-hoc and anytime valid permutation and group invariance testing.
\newblock \emph{arXiv preprint arXiv:2310.01153}, 2023.

\bibitem[Kramkov and Schachermayer(1999)]{MR1722287}
D.~Kramkov and W.~Schachermayer.
\newblock The asymptotic elasticity of utility functions and optimal investment
  in incomplete markets.
\newblock \emph{Ann. Appl. Probab.}, 9\penalty0 (3):\penalty0 904--950, 1999.
\newblock ISSN 1050-5164,2168-8737.
\newblock \doi{10.1214/aoap/1029962818}.
\newblock URL \url{https://doi.org/10.1214/aoap/1029962818}.

\bibitem[Lardy et~al.(2024)Lardy, Gr{\"u}nwald, and
  Harremo{\"e}s]{harremoes2023universal}
Tyron Lardy, Peter Gr{\"u}nwald, and Peter Harremo{\"e}s.
\newblock Reverse information projections and optimal e-statistics.
\newblock \emph{IEEE Transactions on Information Theory (accepted)}, 2024.

\bibitem[Li(1999)]{li1999estimation}
Qiang~Jonathan Li.
\newblock \emph{Estimation of Mixture Models}.
\newblock PhD thesis, Yale University, 1999.

\bibitem[Long(1990)]{LONG199029}
John~B. Long.
\newblock The numeraire portfolio.
\newblock \emph{Journal of Financial Economics}, 26\penalty0 (1):\penalty0
  29--69, 1990.
\newblock ISSN 0304-405X.
\newblock \doi{https://doi.org/10.1016/0304-405X(90)90012-O}.
\newblock URL
  \url{https://www.sciencedirect.com/science/article/pii/0304405X9090012O}.

\bibitem[Platen(2006)]{Platen:2006}
Eckhard Platen.
\newblock A benchmark approach to finance.
\newblock \emph{Math. Finance}, 16\penalty0 (1):\penalty0 131--151, 2006.
\newblock ISSN 0960-1627,1467-9965.
\newblock \doi{10.1111/j.1467-9965.2006.00265.x}.
\newblock URL \url{https://doi.org/10.1111/j.1467-9965.2006.00265.x}.

\bibitem[Ramdas et~al.(2020)Ramdas, Ruf, Larsson, and
  Koolen]{ramdas2020admissible}
Aaditya Ramdas, Johannes Ruf, Martin Larsson, and Wouter Koolen.
\newblock Admissible anytime-valid sequential inference must rely on
  nonnegative martingales.
\newblock \emph{arXiv:2009.03167}, 2020.

\bibitem[Ramdas et~al.(2023)Ramdas, Gr{\"u}nwald, Vovk, and
  Shafer]{ramdas2023game}
Aaditya Ramdas, Peter Gr{\"u}nwald, Vladimir Vovk, and Glenn Shafer.
\newblock Game-theoretic statistics and safe anytime-valid inference.
\newblock \emph{Statistical Science}, 38\penalty0 (4):\penalty0 576--601, 2023.

\bibitem[Robbins(1970)]{robbins1970statistical}
Herbert Robbins.
\newblock Statistical methods related to the law of the iterated logarithm.
\newblock \emph{The Annals of Mathematical Statistics}, 41\penalty0
  (5):\penalty0 1397--1409, 1970.

\bibitem[Shafer(2021)]{shafer2021testing}
Glenn Shafer.
\newblock Testing by betting: A strategy for statistical and scientific
  communication.
\newblock \emph{Journal of the Royal Statistical Society Series A: Statistics
  in Society, with discussion}, 184\penalty0 (2):\penalty0 407--431, 2021.

\bibitem[Vovk and Wang(2021)]{vovk2021values}
Vladimir Vovk and Ruodu Wang.
\newblock E-values: Calibration, combination and applications.
\newblock \emph{The Annals of Statistics}, 49\penalty0 (3):\penalty0
  1736--1754, 2021.

\bibitem[Wasserman et~al.(2020)Wasserman, Ramdas, and
  Balakrishnan]{wasserman2020universal}
Larry Wasserman, Aaditya Ramdas, and Sivaraman Balakrishnan.
\newblock Universal inference.
\newblock \emph{Proceedings of the National Academy of Sciences}, 117\penalty0
  (29):\penalty0 16880--16890, 2020.

\bibitem[Waudby-Smith and Ramdas(2023)]{waudby2020estimating}
Ian Waudby-Smith and Aaditya Ramdas.
\newblock Estimating means of bounded random variables by betting.
\newblock \emph{Journal of the Royal Statistical Society, Series B
  (Methodology), with discussion}, 2023.

\bibitem[Zhang et~al.(2024)Zhang, Ramdas, and Wang]{zhang2023existence}
Zhenyuan Zhang, Aaditya Ramdas, and Ruodu Wang.
\newblock On the existence of powerful p-values and e-values for composite
  hypotheses.
\newblock \emph{Annals of Statistics}, 2024.

\end{thebibliography}
\bibliographystyle{plainnat}

\appendix

\section{A generalized Lebesgue decomposition} \label{app_gen_lebesgue}

The following generalized Lebesgue decomposition into a `regular' and `singular' part is used in the proof of Lemma~\ref{L_u_max}. As always, we work on a given measurable space $(\Omega,\Fcal)$. Given a nonempty subset $\Pcal \subset M_+$, an event $N \in \Fcal$ is called \emph{$\Pcal$-negligible} if $\P(N) = 0$ for all $\P \in \Pcal$.

\begin{lemma} \label{L_gen_Lebesgue}
For any $\Q \in M_+$ and nonempty $\Pcal \subset M_+$ there is a unique decomposition $\Q=\Q^r + \Q^s$ with $\Q^r,\Q^s \in M_+$ such that $\Q^r \ll \Pcal$ and $\Q^r,\Q^s$ are singular, meaning that there exists a $\Pcal$-negligible event $N \in \Fcal$ such that $\Q^s(N^c) = 0$. The event $N$ is essentially unique, i.e., unique up to $\Q$- and $\P$-nullsets for all $\P \in \Pcal$. In particular, any other $\Pcal$-negligible event $N'$ satisfies $\Q(N' \setminus N) = 0$.
\end{lemma}

\begin{proof}
Define $\alpha = \sup_{N \in \Ncal} \Q(N)$ where $\Ncal = \{N \in \Fcal \colon \P(N) = 0 \text{ for all } \P \in \Pcal\}$ is the family of all $\Pcal$-negligible sets. Note that $\Ncal$ is closed under countable unions. Pick a maximizing sequence $N_n$ and set $N = \bigcup_n N_n$. Then $\alpha \ge \Q(N) \ge \Q(N_n) \to \alpha$, so $\Q(N) = \alpha$. Now define $\Q^s(A) = \Q(A \cap N)$ and $\Q^r(A) = \Q(A \cap N^c)$. To see that $\Q^r \ll \Pcal$, suppose this is not the case. Then there is some $N' \in \Ncal$ with $\Q^r(N') > 0$, and hence $\Q(N \cup N') \ge \Q^s(N) + \Q^r(N') > \alpha$, a contradiction. To show uniqueness, suppose we have two tuples $\Q^r_i, \Q^s_i, N_i$ for $i=1,2$ with the stated properties. Set $N = N_1 \cup N_2 \in \Ncal$ and note that $\Q^s_i(N^c) \le \Q^s_i(N_i^c) = 0$. Then for any $A$, we have $\Q(A \cap N) = \Q^s_i(A \cap N) = \Q^s_i(A)$. Thus $\Q^s_1 = \Q^s_2$, and then also $\Q^r_1 = \Q^r_2$. Furthermore, the symmetric difference $N_1 \Delta N_2$ is a $\Q$-nullset because $\Q(N_1 \cap N_2^c) = \Q_2^r(N_1) = 0$ and similarly $\Q(N_2 \cap N_1^c) = 0$. Since also $N_1 \Delta N_2$ belongs to $\Ncal$, the essential uniqueness statement follows.
Finally, if $N' \in \Ncal$ then $\Q(N' \setminus N) = \Q^r(N' \setminus N) + \Q^s(N' \setminus N) \le \Q^r(N') + \Q^s(N^c) = 0$.
\end{proof}

\begin{landscape}

\section{Overview of finiteness properties} \label{app_finiteness}

The table below gives an overview of the relation between absolute continuity properties and finiteness of the numeraire and relative entropies. Each cell shows the possibilities for the property or quantity indicated in the column header, under the assumption indicated for that row. In the table, $X^*$ is the numeraire, $\P^*$ is the RIPr, the relative entropy $H(\Q \mid \P^*)$ is in \eqref{eq_rel_ent_alt}, and the maximal description gain $H(\Q \mid \P^* \leadsto \Peff)$ is defined in Remark~\ref{R_description_gain}. We specifically highlight the possibility that $H(\Q \mid \P)$ may be infinite for all $\P \in \Pcal$, but finite or even zero for the RIPr $\P^* \in \Peff$; this happens, for instance, in Examples~\ref{ex_P_bipolar} and \ref{ex_cauchy_mixture_of_gaussians}. Furthermore, although $H(\Q \mid \P^* \leadsto \Peff) = \infty$ if $\Q \not\ll \Pcal$, the modified quantity in Theorem~\ref{T_duality_general}\ref{T_duality_general_3} will be equal to zero for $\P = \P^*$.

\bigskip

\setlength{\extrarowheight}{3pt} 

\centering
\begin{tabularx}{1.29\textwidth}{|p{5.5cm}|p{3cm}|p{1.4cm}|p{1.4cm}|p{2cm}|p{3cm}X|}
\hline
& \centering $\Q(X^* < \infty) = 1$ & \centering $\P^* \ll \Q$ & \centering $\Q \ll \P^*$ & \centering $H(\Q \mid \P^*)$ & \centering $H(\Q \mid \P^* \leadsto \Peff)$ & \\[3pt] \hline
\centering \vspace{0pt} $\Q \not\ll \Pcal$ \\
    & \centering \vspace{0pt} no & \centering \vspace{0pt} yes & \centering \vspace{0pt} no & \centering \vspace{0pt} $=\infty$ & \centering \vspace{0pt} $=\infty$ & \\ \hline
\centering \vspace{0pt} $\Q \ll \Pcal$ \\
    & \centering \vspace{0pt} yes & \centering \vspace{0pt} yes & \centering \vspace{0pt} yes & \centering \vspace{0pt} $\le \infty$ & \centering \vspace{0pt} $=0$ & \\ \hline
\centering \vspace{0pt} $\Q \ll \Pcal$ \\[1ex] $H(\Q \mid \P) < \infty$ for some $\P \in \Pcal$ \\
    & \centering \vspace{0pt} yes & \centering \vspace{0pt} yes & \centering \vspace{0pt} yes & \centering \vspace{0pt} $< \infty$ & \centering \vspace{0pt} $=0$ & \\ \hline
\centering \vspace{0pt} $\Q \ll \Pcal$ \\[1ex] $H(\Q \mid \P) = \infty$ for all $\P \in \Pcal$ \\
    & \centering \vspace{0pt} yes & \centering \vspace{0pt} yes & \centering \vspace{0pt} yes & \centering \vspace{0pt} $\le \infty$ & \centering \vspace{0pt} $=0$ & \\ \hline
\centering \vspace{0pt} $\Q \ll \Pcal \ll \mu$ for some $\mu$ \\[1ex] $H(\Q \mid \P) = \infty$ for all $\P \in \Pcal$ \\[1ex] $\Pcal$ is convex \\
    & \centering \vspace{0pt} yes & \centering \vspace{0pt} yes & \centering \vspace{0pt} yes & \centering \vspace{0pt} $\le\infty$ & \centering \vspace{0pt} $=0$ & \\ \hline
\end{tabularx}

\end{landscape}

\end{document}